\documentclass{amsart}
\usepackage{graphicx,mathrsfs,setspace,xypic,fancyhdr}
\newtheorem{thm}{Theorem}[section]
\newtheorem{cor}[thm]{Corollary}
\newtheorem{lem}[thm]{Lemma}
\newtheorem{prop}[thm]{Proposition}
\theoremstyle{definition}
\newtheorem{defn}[thm]{Definition}
\theoremstyle{remark}
\newtheorem{rem}[thm]{Remark}
\numberwithin{equation}{section}

\begin{document}

\title{Pfaffian Quartic Surfaces and Representations of Clifford Algebras}
\author{Emre Coskun}
\address{Department of Mathematics, University of Western Ontario, London, ON, N6A 5B7 CANADA}
\email{ecoskun@uwo.ca}
\author{Rajesh S. Kulkarni}
\address{Department of Mathematics, Michigan State University, East Lansing, MI 48824}
\email{kulkarni@math.msu.edu}
\author{Yusuf Mustopa}
\address{Department of Mathematics, University of Michigan, Ann Arbor, MI 48109}
\email{ymustopa@umich.edu}


\date{\today}
\begin{abstract}
Given a nondegenerate ternary form $f=f(x_1,x_2,x_3)$ of degree 4 over an algebraically closed field of characteristic zero, we use the geometry of K3 surfaces and van den Bergh's correspondence between representations of the generalized Clifford algebra $C_f$ associated to $f$ and Ulrich bundles on the surface $X_f:=\{w^{4}=f(x_1,x_2,x_3)\} \subseteq \mathbb{P}^3$ to construct a positive-dimensional family of irreducible representations of $C_f.$

The main part of our construction, which is of independent interest, uses recent work of Aprodu-Farkas on Green's Conjecture together with a result of Basili on complete intersection curves in $\mathbb{P}^{3}$ to produce simple Ulrich bundles of rank 2 on a smooth quartic surface $X \subseteq \mathbb{P}^3$ with determinant $\mathcal{O}_X(3).$  This implies that every smooth quartic surface in $\mathbb{P}^3$ is the zerolocus of a linear Pfaffian, strengthening a result of Beauville-Schreyer on general quartic surfaces.
\end{abstract}
\maketitle
\thispagestyle{plain}

\section{Introduction}

This article lies at the intersection of four topics: representations of generalized Clifford algebras, stable vector bundles on surfaces, linear Pfaffian representations of hypersurfaces, and the Brill-Noether theory of complete intersection curves. 

Let $f=f(x_1, \cdots ,x_n)$ be a nondegenerate homogeneous form of degree $d \geq 2$ in $n \geq 2$ variables over a field $k$ (which we assume throughout to be algebraically closed of characteristic zero). The \textit{generalized Clifford algebra} of $f,$ which we denote by $C_f,$ is the quotient of the free associative $k-$algebra $k\{u_1, u_2, u_3\}$ by the two-sided ideal generated by the relations
\[
\left(\sum_i \alpha_i u_i\right)^d = f(\alpha_1, \cdots ,\alpha_n) \hskip5pt \textnormal{ for all } \alpha_1, \cdots ,\alpha_n \in k.
\]
M. Van den Bergh showed in \cite{vdB} that $dr-$dimensional matrix representations of $C_f$ are in one-to-one correspondence with vector bundles on the degree-$d$ hypersurface $X_f := \{w^d = f(x_1, \cdots ,x_n)\}$ in $\mathbb{P}^n$ whose direct image under the natural linear projection $\pi:X_f \rightarrow \mathbb{P}^{n-1}$ (defined by forgetting the $w-$coordinate) is a trivial vector bundle of rank $dr$ on $\mathbb{P}^{n-1}.$  Since the dimension of any matrix representation of $C_f$ is divisible by $d,$ this accounts for all finite-dimensional representations of $C_f$ (Proposition \ref{divbyd}).

A vector bundle $\mathcal{E}$ satisfies this condition if and only if it is arithmetically Cohen-Macaulay (ACM) with Hilbert polynomial equal to $dr \binom{t + n - 1}{n - 1}$ (Proposition \ref{ulrichacm}). Such vector bundles occur naturally in other algebraic and algebro-geometric contexts (e.g. \cite{ESW, CH, CKM, MP}), and in the literature they are known as {\it Ulrich bundles}.  Since Ulrich bundles are semistable (e.g. Proposition \ref{cliffgiesstab}), it follows that representations of $C_f$ can be parametrized (up to a suitable notion of equivalence) by a union of quasi-projective schemes.

In this paper we are concerned with the case where $f=f(x_1,x_2,x_3)$ is a general nondegenerate ternary quartic form, i.e. where $X_f$ is a smooth quartic surface in $\mathbb{P}^3.$  The Ulrich bundles on $X_f$ corresponding to irreducible representations of $C_f$ are precisely those Ulrich bundles that are stable with respect to $\mathcal{O}_{X_f}(1)$ (Proposition \ref{lem-irred-stable}), so the natural first step in our study is to construct stable Ulrich bundles on $X_f.$  Since one cannot expect $X_f$ to admit an Ulrich line bundle (this essentially follows from Proposition \ref{curvenef}) we look to Ulrich bundles of rank 2, namely those whose first Chern class is a multiple of the hyperplane class.    

If $\mathcal{E}$ is such an Ulrich bundle, its Hilbert polynomial is $8\binom{t+2}{2},$ so we have $c_1(\mathcal{E})=3H$ and $c_2(\mathcal{E})=14.$  The main ingredient in our construction of stable Ulrich bundles on $X_f$ is the following general result.  Recall that a vector bundle is \textit{simple} if its only endomorphisms are scalar multiples of the identity.

\begin{thm}
\label{allquart}
Every smooth quartic surface in $\mathbb{P}^{3}$ admits a 14-dimensional family of simple Ulrich bundles of rank 2 with $c_1=3H$ and $c_2=14.$
\end{thm}

Before turning to its role in constructing stable Ulrich bundles on $X_f,$ we discuss the relevance of this theorem to linear Pfaffian representations of quartic surfaces.  It is known that the existence of a rank-2 Ulrich bundle on a hypersurface $Y \subseteq \mathbb{P}^{n}$ of degree $d \geq 2$  with first Chern class $(d-1)H$ is equivalent to the existence of an $2d \times 2d$ skew-symmetric matrix $M$ of linear forms whose Pfaffian cuts out $Y;$ this follows from Theorem B in \cite{Bea} or Theorem 0.3 in \cite{ESW}.  We then have the following consequence of Theorem \ref{allquart}:

\begin{cor}
\label{allpfaff}
Every smooth quartic surface in $\mathbb{P}^3$ admits a linear Pfaffian representation.
\end{cor}

The existence of linear Pfaffian representations for a general surface of degree $d=4, \cdots ,15$ in $\mathbb{P}^3$ was shown by Beauville-Schreyer (Proposition 7.6(b) of \cite{Bea}) with the help of a Macaulay 2 computation, and a similar method was used by Iliev-Markushevich \cite{IM} to prove the existence for the case of a general quartic threefold in $\mathbb{P}^4.$  However, the proofs do not yield explicit Zariski-open subsets of the respective spaces of hypersurfaces.  A non-computer-assisted proof of the result on general quartic threefolds in \textit{loc.~cit.} was recently found by Brambilla-Faenzi, who showed that a smooth quartic threefold is linear Pfaffian if its Fano variety of lines has a generically reduced component ((ii) of Theorem 5.5 in \cite{BF}).

By the Noether-Lefschetz theorem, the Picard group of a general quartic surface in $\mathbb{P}^3$ is generated by the class of a hyperplane.  Together with Corollary 6.6 of \cite{Bea}, this implies that the general quartic surface in $\mathbb{P}^{3}$ does not admit any Ulrich line bundles.  Any strictly semistable Ulrich bundle is destabilized by a subbundle which is Ulrich (Lemma \ref{cliffdestab}), so we may conclude that the rank-2 Ulrich bundles constructed in Proposition 7.6(b) of \textit{loc.~cit.} are stable.  For a general ternary quartic form $f,$ the Picard number of $X_f$ is 8, so a different approach is required to produce stable Ulrich bundles on $X_f$.  

We combine Theorem \ref{allquart} with an analysis of the effective cone of $X_f$ to prove the following:   

\begin{thm}
\label{main}
If $f=f(x_1,x_2,x_3)$ is a general ternary quartic form, then $X_f$ admits a 14-dimensional family of stable Ulrich bundles of rank 2 with $c_1=3H$ and $c_2=14.$  In particular, there exists a 14-dimensional family of 8-dimensional irreducible representations of the generalized Clifford algebra $C_f.$
\end{thm}

This result is a consequence of the fact that the general simple Ulrich bundle on $X_f$ granted by Theorem \ref{allquart} is stable (Proposition \ref{stable}).  We offer two proofs of this, both of which depend on the rational polyhedral structure of the effective cone of $X_f.$  The first proof proceeds via Proposition \ref{simpstab}, which is based on results of Qin concerning the stability of rank-2 simple bundles on surfaces (Theorem 2 in \cite{Qin1} and Theorem A in \cite{Qin2}).  The second proof starts from the semistability of Ulrich bundles and rules out the existence of destabilizing line subbundles by direct calculation (Proposition \ref{halfstab}).       

The rest of this introduction is devoted to an outline of the proof of Theorem \ref{allquart}, which occupies the bulk of Section \ref{allproofs}.  We believe that the general technique is applicable to the construction of rank-2 Ulrich bundles on K3 surfaces that are not smooth quartics, as well as the construction of rank-2 Ulrich bundles on (special) smooth quartics whose first Chern class is not $3H.$   

We begin by sketching the link between vector bundles on a surface $X$ and the Brill-Noether theory of curves on $X.$  If $\mathcal{E}$ is a globally generated vector bundle of rank $2$ on $X,$ then a general choice of two global sections of $\mathcal{E}$ gives rise to an injective morphism $\mathcal{E}^{\vee} \rightarrow \mathcal{O}_{X}^2$ whose degeneracy locus is a smooth curve $C \in |\det{\mathcal{E}}|$ and which fits into an exact sequence
\begin{equation}
\label{firstelmod}
0 \rightarrow \mathcal{E}^{\vee} \rightarrow \mathcal{O}_{X}^{2} \rightarrow \mathcal{L} \rightarrow 0
\end{equation}
where $\mathcal{L}$ is a line bundle on $C.$  Since $\mathcal{L}$ is globally generated with degree  $c_2(\mathcal{E}),$ its isomorphism class is contained in the Brill-Noether locus $W^1_{c_2(\mathcal{E})}(C).$  Conversely, one can take a basepoint-free pencil associated to a line bundle $\mathcal{L}$ on $C$ and extend its evaluation morphism $\mathcal{O}_C^{2} \rightarrow \mathcal{L}$ by zero to a morphism $\mathcal{O}_X^{2} \rightarrow \mathcal{L}$ whose kernel is a rank-2 vector bundle on $X.$     

This construction has been used extensively to study the interplay between curves and vector bundles on K3 surfaces, starting with \cite{Laz}.  It plays a major role in Voisin's proof of the Green conjecture for generic curves (\cite{Voi1},\cite{Voi2}), and in the recent proof of the Green conjecture for curves on arbitrary K3 surfaces by Aprodu-Farkas \cite{AF}.   
 
One might hope to prove Theorem \ref{allquart} via the following strategy: produce a basepoint-free line bundle $\mathcal{L}$ of degree 14 on a smooth curve $C \in |\mathcal{O}_X(3)|$ and construct an exact sequence of the form (\ref{firstelmod}) where $\mathcal{E}$ is a simple Ulrich bundle.  To implement this strategy or anything like it, we must resolve two issues.  The first is that while there are plenty of line bundles of degree 14 on $C$ with two global sections, the standard results on Brill-Noether loci do not imply that any of them are basepoint-free.  The second is that even if the desired degree-14 line bundle exists on $C$ and the associated rank-2 vector bundle $\mathcal{E}$ on $X$ is simple, it is not necessarily true that $\mathcal{E}$ is Ulrich; see Section \ref{badbundles} for an example.    

Our path to resolving the first issue leads through a result of Aprodu-Farkas (Theorem 3.12 in \cite{AF}) yielding an upper bound of the dimension of certain Brill-Noether loci that allows us to conclude the general element of $W^1_{14}(C)$ is basepoint-free.  In order to apply this result to $C,$ we must know that the Clifford index of $C$ (see Definition \ref{cliffinddef}) is computed by a pencil of minimal degree, and that this minimal degree is sufficiently small relative to the genus of $C.$  

Since a smooth member of $|\mathcal{O}_X(3)|$ is a complete intersection of $X$ with a cubic surface, both of these properties are verified by results of Basili (Th\'{e}or\`{e}me 4.2 and 4.3 in \cite{Bas}).  With a bit more work, we can show that $W^1_{14}(C)$ admits a component $\mathcal{W}$ which is generically smooth of the expected dimension, so that the rank-2 vector bundle $\mathcal{E}$ on $X$ associated to a general $\mathcal{L} \in \mathcal{W}$ is simple; see Proposition \ref{lbexist} for details.    

The resolution of the second issue is somewhat more involved.  While a simple vector bundle $\mathcal{E}$ produced from a general line bundle $\mathcal{L} \in \mathcal{W}$ need not be Ulrich, Proposition \ref{weaklyequiv} implies that it is \textit{weakly Ulrich}.  Such bundles were introduced in \cite{ESW}, and while they are not necessarily ACM, they are a useful generalization of Ulrich bundles; see Definition \ref{weakdef} and the subsequent discussion, as well as Remark \ref{cfweak}.  

To clarify the obstruction to our weakly Ulrich bundles being Ulrich, we pass from the Brill-Noether locus $W^1_{14}(C)$ to the Hilbert scheme $X^{[14]}$ of length-14 subschemes of $X.$  While our bundles might not be globally generated, we may use the sequence (\ref{firstelmod}) to show that each of them may be expressed as a Serre extension
\begin{equation}
\label{zeroserre}
0 \rightarrow \mathcal{O}_X \rightarrow \mathcal{E} \rightarrow \mathcal{I}_{Z|X}(3) \rightarrow 0
\end{equation}
for some l.c.i. $Z \in X^{[14]}$ (Proposition \ref{weakserre}).  We then show that $\mathcal{E}$ is Ulrich precisely when $Z$ does not lie on a quadric (Proposition \ref{quadriculrich}).

Roughly speaking, as $\mathcal{E}$ varies in our family, the subschemes $Z$ which appear in (\ref{zeroserre}) sweep out a 21-dimensional locus in $X^{[14]}.$  Due to the fact that $\mathcal{E}$ is locally free, any such $Z$ is Cayley-Bacharach with respect to $\mathcal{O}_X(3).$  To show that the general $Z$ in our 21-dimensional locus does not lie on a quadric, we prove that locus of length-14 subschemes of $X$ which lie on a quadric and are also Cayley-Bacharach with respect to $\mathcal{O}_X(3)$ is of dimension at most 20.  All this is explained in more detail in the proofs of Propositions \ref{quadriccb} and \ref{bigdim}.

\bigskip

\textbf{Acknowledgments:} The second author was partially supported by the NSF grants DMS-0603684 and DMS-1004306. The third author was supported by the NSF grant RTG DMS-0502170.  We are grateful to M. Deland, I. Dolgachev, and D. Faenzi for useful discussions related to this work, and to R. Hartshorne for pointing us towards an error in a previous version of this manuscript.  We would also like to thank R. Lazarsfeld for his important role in forming our collaboration and for useful discussions.

\section{Preliminaries}

\subsection{Representations of Clifford Algebras}
\label{repcliff}
We give a summary of van den Bergh's correspondence.  Proofs will be omitted; we refer to van den Bergh's original article \cite{vdB} for more details.

For the rest of this section we fix a homogeneous form $f=f(x_{1}, \cdots ,x_{n})$ of degree $d \geq 2$ which is \textit{nondegenerate}, i.e. satisfies the property that $\frac{\partial{f}}{\partial{x_{i}}}(y_{1}, \cdots ,y_{n})=0$ for all $i$ if and only if $y_{j}=0$ for all $j.$

\begin{defn}
The \emph{Clifford algebra} $C_f$ of the form $f$ is the associative $k$-algebra $k\{x,y,z\}/I$, where $k\{x_1, \cdots ,x_n\}$ is the free associative $k-$algebra on $x_1, \cdots ,x_n$ and $I \subseteq k\{x_1, \cdots ,x_n\}$ is the two-sided ideal generated by elements of the form $(\alpha_1{x_1} + \cdots  + \alpha_n{x_n})^d-f(\alpha_1, \cdots ,\alpha_n)$ for $\alpha_1, \cdots, \alpha_n \in k$.
\end{defn}

\begin{defn}
Let $C_{f}$ be the Clifford algebra associated to $f.$
\begin{itemize}
\item[(i)]{A \textit{representation of $C_{f}$} is a $k-$algebra homomorphism $\phi: C_{f} \rightarrow \textnormal{Mat}_{m}(k).$  The positive integer $m$ is the \textit{dimension} of $\phi.$}
\item[(ii)]{Two $m-$dimensional representations $\phi_{1},\phi_{2}$ of $C_{f}$ are \textit{equivalent} if there exists an invertible $\theta \in \textnormal{Mat}_{m}(k)$ such that $\phi_{1}=\theta\phi_{2}\theta^{-1}.$}
\end{itemize}
\end{defn}

Observe that the data of an $m-$dimensional representation $\phi$ of $C_{f}$ consists of a collection $A_{1}, \cdots ,A_{n}$ of $m \times m$ matrices for which the identity
\begin{equation}
\label{linearisation}
(x_{1}A_{1} + \cdots + x_{n}A_{n})^{d}=f(x_{1}, \cdots ,x_{n})I_{m}
\end{equation}
holds in $\textnormal{Mat}_{m}(k[x_{1}, \cdots ,x_{n}]).$

Any study of representations begins with the study of irreducible representations, for which we now give an appropriate definition.

\begin{defn}
\label{irdef}
A representation $\phi:C_{f} \rightarrow \textnormal{Mat}_{m}(k)$ is \textit{irreducible} if its image $\phi(C_{f})$ generates $\textnormal{Mat}_{m}(k).$  Otherwise we say that $\phi$ is \textit{reducible}.
\end{defn}




Representations of Clifford algebras have been studied in \cite{BHS}, where they are referred to as \textit{Clifford modules.}  It is observed in \textit{loc.~cit.}~that $C_{f}$ admits a natural $\mathbb{Z}/d\mathbb{Z}$-grading, and that any representation of $C_{f}$ admits the structure of a $\mathbb{Z}/d\mathbb{Z}-$graded $C_{f}-$module.  In light of this, the following statement, which is an immediate consequence of Corollary 2 in \cite{vdB} or Proposition 1.1 in \cite{HT}, is very natural.

\begin{prop}
\label{divbyd}
The dimension of any representation of $C_{f}$ is equal to $rd$ for some positive integer $r.$ \hfill \qedsymbol
\end{prop}

We now describe how to associate to each $rd-$dimensional representation $\phi$ of $C_{f}$ a vector bundle of rank $r$ on the hypersurface $X_f \subseteq \mathbb{P}^{n}$ defined by $w^d=f(x_1,\ldots,x_n)$.  Define a $k-$algebra homomorphism 
\begin{equation}
\Phi: k[w,x_{1}, \cdots ,x_{n}] \rightarrow \textnormal{Mat}_{rd}(k[x_{1}, \cdots ,x_{n}])
\end{equation}
by setting $\Phi(x_{i})=x_{i} \cdot I_{rd}$ for $i=1, \cdots, n$ and $\Phi(w)=\sum_{i=1}^{n}x_{i} \cdot A_{i}.$  By (\ref{linearisation}) this descends to a homomorphism $\overline{\Phi}: S_{X_f} \rightarrow \textnormal{Mat}_{rd}(k[x_{1}, \cdots ,x_{n}]).$ where $S_{X_f}$ is the homogeneous coordinate ring of $X_f.$  This yields an $S_{X_f}-$module structure on $k[x_{1}, \cdots ,x_{n}]^{rd}.$

Since composing $\overline{\Phi}$ with the natural inclusion $k[x_{1}, \cdots ,x_{n}] \hookrightarrow S_{X_f}$ yields the natural $k[x_{1}, \cdots ,x_{n}]-$module structure on $k[x_{1}, \cdots ,x_{n}]^{rd},$ the geometric content of our discussion may be summarized as follows: the homomorphism $\overline{\Phi}$ yields an $\mathcal{O}_{X_f}-$module $\mathcal{E}$ for which $\pi_{\ast}\mathcal{E} \cong \mathcal{O}_{\mathbb{P}^{n-1}}^{rd},$ where $\pi:X_f \rightarrow \mathbb{P}^{n-1}$ is the projection which forgets the variable $w.$  It can be shown that the nondegeneracy hypothesis on $f$ implies that $\mathcal{E}$ is locally free. The main result of \cite{vdB} (Proposition 1 in \textit{loc.~cit.}) implies that this construction yields an essentially bijective correspondence.

\begin{prop}
\label{vdbcor}
There is a one-to-one correspondence between equivalence classes of $dr-$dimensional representations of $C_{f}$ and isomorphism classes of vector bundles $\mathcal{E}$ of rank $r$ on the hypersurface $X_f$ such that $\pi_{\ast}\mathcal{E} \cong \mathcal{O}_{\mathbb{P}^{n-1}}^{dr}.$ \hfill \qedsymbol
\end{prop}

\subsection{Ulrich Bundles}
\label{ub}
As mentioned in the introduction, the vector bundles that occur on the geometric side of van den Bergh's correspondence are ubiquitous enough to justify the following definition.

\begin{defn}
A vector bundle $\mathcal{E}$ of rank $r$ on $X$ is \textit{Ulrich} if for some (resp. all) linear projections $\pi:X \rightarrow \mathbb{P}^{n-1},$ we have that $\pi_{\ast}\mathcal{E} \cong \mathcal{O}_{\mathbb{P}^{n-1}}^{dr}.$
\end{defn}

The following characterization of Ulrich bundles plays an important role in the sequel.  It is a special case of Proposition 2.1 in \cite{ESW} which we state without proof. 
\begin{prop}
\label{ulrichacm}
Let $\mathcal{E}$ be a vector bundle $\mathcal{E}$ of rank $r \geq 1$ on $X.$  Then the following are equivalent:
\begin{itemize}
\item[(i)]{$\mathcal{E}$ is Ulrich.}
\item[(ii)]{$\mathcal{E}$ is ACM and its Hilbert polynomial is $dr\binom{t+n-1}{n-1}.$}
\item[(iii)]{The $\mathcal{O}_{\mathbb{P}^{n}}-$module $\mathcal{E}$ admits a minimal graded free resolution of the form
\begin{equation}
0 \longrightarrow \mathcal{O}_{\mathbb{P}^{n}}(-1)^{dr} \longrightarrow \mathcal{O}_{\mathbb{P}^{n}}^{dr} \longrightarrow \mathcal{E} \longrightarrow 0.
\end{equation}  \hfill \qedsymbol}
\end{itemize}
\end{prop}
In particular, $X$ is linear determinantal if and only if $X$ admits an Ulrich line bundle.  As an immediate consequence of (iii), we obtain
\begin{cor}
\label{ggsect}
Any Ulrich bundle of rank $r$ on $X$ is globally generated and has $dr$ global sections. \hfill \qedsymbol
\end{cor}

Since most hypersurfaces do not admit linear determinantal representations, one can ask for the ``next best thing," namely a linear Pfaffian representation.  The following result, which allows us to deduce Theorem \ref{allpfaff} from Theorem \ref{allquart}, is a rephrasing of Corollary in \cite{Bea} which is suitable for our purposes.

\begin{prop}
\label{beapfaff}
The following conditions are equivalent:
\begin{itemize}
\item[(i)]{$X$ is the zerolocus of the Pfaffian of a $(2d) \times (2d)$ skew-symmetric matrix of linear forms.}
\item[(ii)]{There exists a rank-2 Ulrich bundle on $X$ with determinant $\mathcal{O}_{X}(d-1).$ \hfill \qedsymbol}
\end{itemize} 
\end{prop}
Note that condition (i) is satisfied if $X$ is linear determinantal.  Indeed, if $M$ is a $d \times d$ matrix of linear forms, we have the equation
\medskip
\begin{small}
\[ \det{M} = {\rm Pf}\left[ \begin{array}{cc}
0 & M \\
-M^{T} &0  \end{array} \right]\] 
\end{small}
\subsection{Weakly Ulrich Bundles}
\label{wub}
In general, it is difficult to produce Ulrich bundles directly.  As a first approximation one can try to produce bundles that are "almost" Ulrich in the following sense.

\begin{defn}
\label{weakdef}
A vector bundle $\mathcal{E}$ on $X$ is called \textit{weakly Ulrich} if $H^j(\mathcal{E}(-m))=0$ for $1 \leq j \leq n-1, m \leq j-1$ and $0 \leq j \leq n-2, m \geq j+2.$
\end{defn}

The importance of this notion stems from the fact that the Beilinson monad of $\mathcal{E}$ (viewed as a coherent sheaf on $\mathbb{P}^n$) reduces to a two-term complex if and only if $\mathcal{E}$ is weakly Ulrich (see Section 2 of \cite{ESW}).  Moreover, the sole nonzero morphism in this two-term complex is a matrix of linear forms if and only if $\mathcal{E}$ is Ulrich.   

The next proposition is an immediate consequence of the fact that the vanishing of cohomology groups is an open condition on families of coherent sheaves. 

\begin{prop}
\label{opencond}
Let $M$ be a family of isomorphism classes of vector bundles on $X.$  Then the locus in $M$ parametrizing Ulrich (resp. weakly Ulrich) bundles is a Zariski-open subset of $M$ (possibly empty). \hfill \qedsymbol
\end{prop}

\subsection{Semistability, Simplicity, and Moduli}
\label{sandm}

We begin by recalling some basic definitions.

\begin{defn}
\label{giesstabdef}
If $\mathcal{G}$ is a torsion-free sheaf on $X$ of rank $r,$ the \textit{reduced Hilbert polynomial of} $\mathcal{G}$ is $p(\mathcal{G}):=\frac{1}{\textnormal{rk}(\mathcal{G})} \cdot H_{\mathcal{G}}(t),$ where $H_{\mathcal{G}}(t)$ is the Hilbert polynomial of $\mathcal{G}.$
\end{defn}

\begin{defn}
A torsion-free sheaf $\mathcal{E}$ of rank $r$ on $X$ is \textit{semistable} (resp. \textit{stable}) if for every subsheaf $\mathcal{F}$ of $\mathcal{E}$ for which $0 < \textnormal{rank}(\mathcal{F}) < r$ we have that (w.r.t. lexicographical order)
\begin{equation}
p(\mathcal{F}) \leq p(\mathcal{E}) \hskip10pt \textnormal{  (resp. }p(\mathcal{F}) < p(\mathcal{E})).
\end{equation}
\end{defn}

The proofs of the following results may be found in \cite{CKM2}.

\begin{prop}
\label{cliffgiesstab}
Let $\mathcal{E}$ be an Ulrich bundle of rank $r \geq 1$ on $X.$  Then $\mathcal{E}$ is semistable. \hfill \qedsymbol
\end{prop}

\begin{lem}
\label{cliffdestab}
Let $\mathcal{E}$ be an Ulrich bundle on $X$ of rank $r$ which is strictly semistable.  Then there exists a subbundle $\mathcal{F}$ of $\mathcal{E}$ having rank $s < r$ which is Ulrich. \hfill \qedsymbol
\end{lem}

\begin{prop}
\label{lem-irred-stable}
Let $f$ be a nondegenerate homogeneous form of degree $d \geq 2$ in $n$ variables, and let $X_f \subseteq \mathbb{P}^{n}$ be the smooth hypersurface defined by the equation $w^{d}=f.$  If $\mathcal{E}$ is an Ulrich bundle on $X_f,$ then the representation of the Clifford algebra $C_{f}$ associated to $\mathcal{E}$ is irreducible if and only if $\mathcal{E}$ is stable. \hfill \qedsymbol
\end{prop}

The final result in the subsection, which is implied by part (1) of Theorem 0.1 in \cite{Muk}, is important for the proof of Theorem \ref{allquart}.

\begin{thm}
\label{simpsmooth}
Let $X$ be a K3 surface, and denote by ${\rm Spl}(r,c_1,c_2)$ the moduli space of simple vector bundles on $X$ of rank $r$ with first and second Chern classes $c_1$ and $c_2,$ respectively.  If ${\rm Spl}(r,c_1,c_2)$ is nonempty, then it is smooth, and its dimension at each point $[\mathcal{E}]$ is equal to $1-\chi(\mathcal{E} \otimes \mathcal{E}^{\vee}).$  \hfil \qedsymbol
\end{thm}

\section{Ulrich Bundles on Quartic Surfaces}
\label{allproofs}
\subsection{Some results on the geometry of curves}
Before starting the proof of Theorem \ref{allquart} in earnest, we gather the results on curves which play a major role in the sequel.  Throughout this subsection, $C$ is a smooth projective curve of genus $g \geq 1$.

We begin by defining two important invariants of a curve.

\begin{defn}
\label{gondef}
The gonality of $C,$ which we denote by ${\rm gon}(C),$ is the minimum degree of a finite morphism from $C$ to $\mathbb{P}^1$.
\end{defn}

Observe that any finite morphism $f:C \rightarrow \mathbb{P}^1$ which computes ${\rm gon}(C)$ is induced by a complete basepoint-free linear series on $C.$

\begin{defn}
\label{cliffinddef}
The Clifford index of $C,$ which we denote by ${\rm Cliff}(C),$ is
\begin{equation}
\min\{d-2r : \ {\rm there \ exists}  \ \mathcal{L} \in \textnormal{Pic}^{d}(C) \ {\rm such \ that} \  h^0(\mathcal{L}) = r+1 \geq 2, h^1(\mathcal{L}) \geq 2\}
\end{equation}
\end{defn}  

By Clifford's Theorem, ${\rm Cliff}(C) \geq 0$ for every curve $C,$ and ${\rm Cliff}(C)=0$ precisely when $C$ is hyperelliptic.  Moreover, it is a straightforward consequence of the definitions that for all $g \geq 1,$ we have 
\begin{equation}
\label{cliffgon}
{\rm Cliff}(C) \leq {\rm gon}(C)-2.
\end{equation}
The following theorem of Basili yields an essentially complete description of these invariants for complete intersection curves in $\mathbb{P}^3.$  

\begin{thm} 
\label{basili}
\cite{Bas} Let $C$  be a smooth, nondegenerate complete intersection curve in ${\bf P}^3$ and let $\ell$ be the maximum number of collinear points on $C$.  Then:
\begin{itemize}
\item[(i)]{[Th\'{e}or\`{e}me 4.2, \textit{loc.~cit.}]  ${\rm gon}(C) = \deg C - \ell$, and an effective divisor $\Gamma \subset C$ computes this gonality if and only if $\Gamma$ is residual, in a plane section of $C$, to a set of $\ell$ collinear points of $C$.}
\item[(ii)]{[Th\'{e}or\`{e}me 4.3, \textit{loc.~cit.}]  If $\deg{C} \neq 9,$ then the ${\rm Cliff}(C)={\rm gon}(C)-2.$ \hfill \qedsymbol}
\end{itemize} 
\end{thm}

\noindent We now turn to varieties of special linear series on $C.$  Our treatment will be very brief; we refer to Chapters III,IV, and V of \cite{ACGH} for details.  For integers $r,d \geq 2,$ the \textit{Brill-Noether locus}
\begin{displaymath}
W^r_d(C) = \{L \in {\rm Pic}^d(C) | h^0(C, L) \geq r + 1\}.
\end{displaymath}
naturally admits the structure of a determinantal subscheme of the Picard variety ${\rm Pic}^d(C)$.  We have the following formula for the expected dimension of these varieties:

\begin{thm}
\label{expdim}
[Theorem 1.1 in Chapter V of \textit{loc.~cit.}] Let $\rho=\rho(g,r,d)$ be the Brill-Noether number $\rho: = g - (r + 1)(g + r - d)$. If $g + r - d \geq 0$, then $W^r_d(C)$ is nonempty.  Furthermore, each irreducible component of $W^r_d(C)$ has dimension at least $\rho$.  
\end{thm}

Note that $W^{r + 1}_d(C) \subseteq W^r_d(C)$ for all $r \geq 0$ and all $d.$  We will need to know that this inclusion is strict in the cases of interest to us.
\begin{lem}
\label{strictinc}
[Lemma 3.5 in Chapter IV of \textit{loc.~cit.}] Suppose $g-d+r \geq 0.$  Then no component of $W^r_d(C)$ is entirely contained in $W^{r+1}_d(C).$ \hfill \qedsymbol
\end{lem}  

In order to ensure that the vector bundles $\mathcal{E}$ we construct using (\ref{firstelmod}) are simple, we will need the following result, which is based on the calculation of the tangent space to $W^r_d(C)$ at a point $L \in W^{r + 1}_d(C)$.

\begin{prop}
\label{tanspace}
[Proposition 4.2 (i) in Chapter IV of \textit{loc.~cit.}]
Let $L$ be a point of $W^{r}_d(C)$ which is not contained in $W^{r+1}_d(C).$  Then $W^r_d(C)$ is smooth of the expected dimension $\rho$ at $L$ if and only if the cup-product map
\begin{equation}
\mu : H^0(L) \otimes H^0(\omega_{C} \otimes L^{-1}) \rightarrow H^0(\omega_C)
\end{equation} 
is injective. \hfill \qedsymbol
\end{prop}
\noindent 


The most recent result in this subsection, which is due to Aprodu-Farkas, implies that Green's Conjecture on syzygies of canonical curves holds for any smooth curve $C$ on a K3 surface which achieves equality in (\ref{cliffgon}); see \cite{AF} for details. 

\begin{thm}
\label{apfarthm}
[Theorem 3.2, \textit{loc.~cit.}] Let $C$ be a smooth curve on a $K3$ surface with ${\rm gon}(C)=k, {\rm Cliff}(C)=k-2,$ and $\rho(g,1,k) \leq 0$ such that the linear system $|C|$ is basepoint-free. If $d \leq g - k + 2$, then every irreducible component of $W^1_d(C)$ has dimension at most $d-k.$ \hfill \qedsymbol
\end{thm}

In Section \ref{badbundles}, it will be useful for us to have the following upper bound on the dimension of $W^r_d(C).$ 

\begin{thm}
\label{martens}
[Theorems 5.1 and 5.2 in Chapter IV of \cite{ACGH}] Let $d,r$ be integers satisfying $2 \leq d \leq g-2, \hskip5pt d \geq 2r > 0.$  If $C$ is neither hyperelliptic, trigonal, bielliptic nor a smooth plane quintic, then every component of $W^r_d(C)$ has dimension at most $d-2r-2.$ \hfill \qedsymbol
\end{thm} 

\subsection{Constructing weakly Ulrich bundles from line bundles on curves}
For the rest of this section, $X$ denotes a smooth quartic surface in $\mathbb{P}^3.$  Our first result gives a list of sufficient conditions on a degree-14 line bundle $\mathcal{L}$ for the rank-2 vector bundle $\mathcal{E}$ in (\ref{firstelmod}) to be simple, weakly Ulrich and globally generated in codimension 1.

\begin{prop}
\label{weaklyequiv}
Let $C \in |\mathcal{O}_X(3)|$ be a smooth curve, and let $\mathcal{L}$ be a line bundle of degree 14 on $C$ satisfying the following conditions:
\begin{itemize}
\item[(i)]{$\mathcal{L}$ is basepoint-free.}
\item[(ii)]{$h^0(\mathcal{L})=2.$}
\item[(iii)]{$h^0(\mathcal{L}(-1))=0.$}
\item[(iv)]{$h^0(\mathcal{L}^{-2}(3))=0.$}
\end{itemize}
Then the rank-2 vector bundle $\mathcal{E}$ constructed from the sequence
\begin{equation}
\label{elmod}
0 \rightarrow \mathcal{E}^{\vee} \rightarrow H^0(\mathcal{L}) \otimes \mathcal{O}_X \rightarrow \mathcal{L} \rightarrow 0
\end{equation}
satisfies the following properties:
\begin{itemize}
\item[(a)]{$c_1(\mathcal{E})=3H$ and $c_2(\mathcal{E})=14.$}
\item[(b)]{$\mathcal{E}$ is weakly Ulrich.}
\item[(c)]{$\mathcal{E}$ is globally generated in codimension 1.}
\item[(d)]{$\mathcal{E}$ is simple.}
\end{itemize}
\end{prop}

\begin{proof}
A Chern class computation applied to (\ref{elmod}) shows that $c_1(\mathcal{E})=3H$ and $c_2(\mathcal{E})=14,$ i.e. that property (a) is satisfied.  To show that property (b) is satisfied, we need to check the following vanishings:
\begin{itemize}
\item[(0)]{$H^0(\mathcal{E}(m))=0$ for $m \leq -2$}
\item[(1)]{$H^1(\mathcal{E}(m))=0$ for $m \leq -3$ and $m \geq 0$}
\item[(2)]{$H^2(\mathcal{E}(m))=0$ for $m \geq -1$}.
\end{itemize}
Since $\mathcal{E}$ is of rank 2 with determinant $\mathcal{O}_X(3),$ we have for each $m \in \mathbb{Z}$ that $\mathcal{E}(m) \cong \mathcal{E}^{\vee}(m+3),$ so Serre duality implies that
\begin{equation}
\label{serreisoms}
H^i(\mathcal{E}(m)) \cong H^i(\mathcal{E}^{\vee}(m+3)) \cong H^{2-i}(\mathcal{E}(-m-3))^{\ast}
\end{equation}
In particular, (0) and (2) are equivalent, and the vanishing of $H^1(\mathcal{E}(m))$ for all $m \leq -3$ is equivalent to its vanishing for all $m \geq 0.$  Therefore we only need to verify (0) and the vanishing of $H^1(\mathcal{E}(m))$ for $m \geq 0.$

Dualizing (\ref{elmod}) gives the sequence
\begin{equation}
\label{dualtwist1}
0 \rightarrow H^0(\mathcal{L})^{\ast} \otimes \mathcal{O}_{X} \rightarrow \mathcal{E} \rightarrow \mathcal{L}^{-1}(3) \rightarrow 0
\end{equation}

To show that (0) holds, it suffices to check that $H^0(\mathcal{E}(-2))=0.$  Twisting (\ref{dualtwist1}) by -2 and taking cohomology, we have that $H^0(\mathcal{E}(-2))$ is contained in $H^0(\mathcal{L}^{-1}(1)).$  Since $\mathcal{L}^{-1}(1)$ is of degree -2, we have $H^0(\mathcal{L}^{-1}(1))=0,$ so $H^0(\mathcal{E}(-2))=0$ as claimed.

To show that $H^1(\mathcal{E}(m))=0$ for $m \geq 0,$ it is more convenient to work with (\ref{elmod}).  Another application of Serre duality, together with taking cohomology of the $(-m)$-th twist of (\ref{elmod}), shows that
\begin{equation}
\label{cokerisom}
H^1(\mathcal{E}(m))^{\ast} \cong H^1(\mathcal{E}^{\vee}(-m)) \cong \textnormal{coker}(H^0(\mathcal{L}) \otimes H^0(\mathcal{O}_X(-m)) \rightarrow H^0(\mathcal{L}(-m)))
\end{equation}
The multiplication map on the right-hand side is an isomorphism when $m=0,$ and when $m > 0,$ we have that $H^0(\mathcal{L}(-m))=0$ by (iii).  In any case, it follows that the map is surjective, i.e. that $h^1(\mathcal{E}(m))=0.$  This concludes the proof that $\mathcal{E}$ satisfies property (b).

As for property (c), (\ref{dualtwist1}) implies that $\mathcal{E}$ is globally generated away from the base locus of $\mathcal{L}^{-1}(3),$ e.g. that it is globally generated in codimension 1.  It remains to show that $\mathcal{E}$ is simple.

The fact that $\mathcal{E}$ is weakly Ulrich implies that $h^0(\mathcal{E}^{\vee})=h^2(\mathcal{E})=0$ and $h^1(\mathcal{E}^{\vee})=h^1(\mathcal{E})=0.$  Twisting (\ref{dualtwist1}) by $\mathcal{E}^{\vee}$ and taking cohomology, we obtain the isomorphism
\begin{equation}
\label{endisom1}
H^{0}(\mathcal{E} \otimes \mathcal{E}^{\vee}) \cong H^{0}(\mathcal{E}^{\vee}|_{C} \otimes \mathcal{L}^{-1}(3))
\end{equation}
Restricting (\ref{elmod}) to $C$ yields the long exact sequence
\begin{equation}
0 \rightarrow \mathcal{L}(-3) \rightarrow \mathcal{E}^{\vee}|_{C} \rightarrow H^{0}(\mathcal{L}) \otimes \mathcal{O}_{C} \rightarrow \mathcal{L} \rightarrow 0
\end{equation}
We may split this into the two short exact sequences
\begin{equation}
\label{split1}
0 \rightarrow \mathcal{L}(-3) \rightarrow \mathcal{E}^{\vee}|_{C} \rightarrow \mathcal{L}^{-1} \rightarrow 0
\end{equation}
\begin{equation}
\label{split2}
0 \rightarrow \mathcal{L}^{-1} \rightarrow H^0(\mathcal{L}) \otimes \mathcal{O}_C \rightarrow \mathcal{L} \rightarrow 0
\end{equation}
Twisting (\ref{split1}) by $\mathcal{L}^{-1}(3),$ we now have
\begin{equation}
0 \rightarrow \mathcal{O}_{C} \rightarrow \mathcal{E}^{\vee}|_{C} \otimes \mathcal{L}^{-1}(3) \rightarrow \mathcal{L}^{-2}(3) \rightarrow 0
\end{equation}
Since $H^0(\mathcal{L}^{-2}(3))=0$ by (iv), it follows that $h^0(\mathcal{E}^{\vee} \otimes \mathcal{L}^{-1}(3))=1.$  By (\ref{endisom1}), $\mathcal{E}$ is simple.
\end{proof}

The next result shows that the hypotheses of Proposition \ref{weaklyequiv} are not vacuous.

\begin{prop}
\label{lbexist}
If $C \in |\mathcal{O}_X(3)|$ is a general smooth curve, then there exists an irreducible component $\mathcal{W}$ of $W^1_{14}(C)$ whose general member satisfies conditions (i),(ii),(iii), and (iv) in the statement of Proposition \ref{weaklyequiv}.  Moreover, $\mathcal{W}$ is generically smooth of the expected dimension $\rho(19,1,14)=7.$
\end{prop}

\begin{proof}
Let us fix a general smooth curve $C \in |\mathcal{O}_{X}(3)|$ for the duration of the proof.  We will first show that every irreducible component of $W^1_{14}(C)$ contains a nonempty Zariski-open subset whose members satisfy conditions (i),(ii) and (iii).  Then we will produce a member $\mathcal{L}_0$ of $W^1_{14}(C)$ which satisfies condition (iv).  It then follows that if $\mathcal{W}$ is any irreducible component of $W^1_{14}(C)$ containing $\mathcal{L}_0,$ then the general member of $\mathcal{W}$ satisfies conditions (i),(ii),(iii), and (iv).  Finally, we will be able to deduce from Proposition \ref{tanspace} that $\mathcal{W}$ is generically smooth of the expected dimension 7. 

By Theorem \ref{expdim}, each irreducible component of $W^1_{14}(C)$ has dimension at least 7.  Since $19-14+1 > 0,$ Lemma \ref{strictinc} implies that each such component contains a nonempty Zariski-open subset whose members satisfy condition (ii).  

The members of $W^1_{14}(C)$ which possess a base point are parametrized by the image of the addition map $\sigma: C \times W^1_{13}(C) \rightarrow W^1_{14}(C)$ which assigns each pair $(p,\mathcal{L}')$ to $\mathcal{L}'(p).$  We will now show that the image of $\sigma$ has dimension at most 6, which implies that the general member of $W^1_{14}(C)$ satisfies condition (i).

The restriction map $H^0(\mathcal{O}_{\mathbb{P}^{3}}(3)) \rightarrow H^0(\mathcal{O}_X(3))$ is an isomorphism, so we may assume that $C$ is the intersection of $X$ with a smooth cubic surface $Y \subseteq \mathbb{P}^3.$  Since $Y$ contains a line, there exist 4 collinear points on $C.$  If there were 5 or more collinear points on $C,$ then by B\'{e}zout's theorem the intersection $X \cap Y$ would contain a line; this contradicts the fact that $X \cap Y=C$ is smooth and irreducible.  Given that the degree of $C$ is 12, Theorem \ref{basili} implies that ${\rm gon}(C)=8$ and ${\rm Cliff}(C)=6.$  Furthermore, $\rho(19,1,8)<0,$ so $C$ satisfies the hypotheses of Theorem \ref{apfarthm}; consequently the dimension of $W^1_{13}(C)$ is at most 5.  This implies in turn that the image of $\sigma$ has dimension at most 6, e.g. that the general member of $W^1_{14}(C)$ satisfies condition (i).  

Given that $\mathcal{O}_C(1)$ has degree 12, $\mathcal{L}(-1)$ is of degree 2 for all $\mathcal{L} \in {\rm Pic}^{14}(C).$  Line bundles of degree 2 having global sections form a 2-dimensional family, so the general member of $W^1_{14}(C)$ satisfies condition (iii).

At this stage, we have not ruled out the possibility that $h^{0}(\mathcal{L}^{-2}(3)) \neq 0$ for all $\mathcal{L} \in W^1_{14}(C),$ i.e. that every member of $W^1_{14}(C)$ fails to satisfy condition (iv).  The set of all $\mathcal{L} \in \textnormal{Pic}^{14}(C)$ for which $h^0(\mathcal{L}^{-2}(3)) \geq 1$ is Zariski-closed, so to show that there exists an irreducible component $\mathcal{W}$ of $W^1_{14}(C)$ whose general member satisfies condition (iv) (we have already checked conditions (i) through (iii)) it is enough to exhibit an element $\mathcal{L}_0$ of $W^1_{14}(C)$ for which $h^0(\mathcal{L}_0^{-2}(3))=0$; we may then take $\mathcal{W}$ to be any irreducible component of $W^1_{14}(C)$ containing $\mathcal{L}_0.$     

Let $p,q \in C$ be sufficiently general so that the tangent lines to $C \subseteq \mathbb{P}^{3}$ at $p$ and $q$ are not coplanar (i.e. do not intersect).  Then the line bundle $\mathcal{L}_0:=\mathcal{O}_{C}(1)(p+q)$ is a member of $W^1_{14}(C),$ since its degree is 14 and $h^0(\mathcal{L}_0) \geq h^0(\mathcal{O}_C(1)) > 2.$  Our hypotheses on $p$ and $q$ imply that $h^0(\mathcal{L}_0^{-2}(3))=h^0(\mathcal{O}_{C}(1)(-2p-2q))=0,$ so $\mathcal{L}_0$ satisfies condition (iv).  

Take $\mathcal{W}$ to be an irreducible component of $W^1_{14}(C)$ containing $\mathcal{L}_0,$ and let $\mathcal{L}$ be a general element of $\mathcal{W}.$  The previous arguments imply that $\mathcal{L}$ satisfies conditions (i),(ii), (iii) and (iv), so $\mathcal{L}$ fits into an exact sequence of the form (\ref{split2}).  Twisting (\ref{split2}) by $\mathcal{L}^{-1}(3)$ and taking cohomology, we see that the kernel of the cup-product map $\mu : H^0(\mathcal{L}) \otimes H^0(\omega_{C} \otimes \mathcal{L}^{-1}) \rightarrow H^0(\omega_{C})$ is isomorphic to $H^0(\mathcal{L}^{-2}(3)).$  Since the latter cohomology group is zero, it follows from Theorem \ref{tanspace} that $\mathcal{W}$ is smooth of dimension 7 at $\mathcal{L}.$  This concludes the proof. 
\end{proof}

\begin{rem}
Theorem \ref{tanspace} does not apply to the line bundle $\mathcal{L}_0$ constructed in the previous proof, since $\mathcal{L}_0$ is a member of $W^2_{14}(C).$ 
\end{rem}
\begin{cor}
\label{existssimple}
There exists a simple weakly Ulrich bundle $\mathcal{E}$ rank 2 on $X$ with $c_1(\mathcal{E})=3H$ and $c_2(\mathcal{E})$=14 which is globally generated in codimension 1. \hfill \qedsymbol
\end{cor}

\begin{prop}
\label{simplemoduli}
There exists a smooth irreducible component  $\textnormal{Spl}_{X}^{\circ}(2,3H,14)$ of the moduli space $\textnormal{Spl}_{X}(2,3H,14)$ of simple rank-2 vector bundles on $X$ with $c_1=3H,c_2=14$ and a nonempty Zariski-open subset $\mathcal{U} \subseteq \textnormal{Spl}^{\circ}_{X}(2,3H,14)$ parametrizing weakly Ulrich bundles on $X.$  Moreover, we have a diagram
\begin{equation}
\label{grassdiag}
\xymatrix{
        &\mathbb{G}_{\mathcal{U}} \ar@{-->}[r]^{h} \ar[d]_{p}  &|\mathcal{O}_X(3)| \\
&\mathcal{U}    &
}
\end{equation}
where $p: \mathbb{G}_{\mathcal{U}} \rightarrow \mathcal{U}$ is the Grassmann bundle whose fiber over $\mathcal{E} \in \mathcal{U}$ is $G(2,H^0(\mathcal{E}))$ and $h: \mathbb{G}_{\mathcal{U}} \dashrightarrow |\mathcal{O}_X(3)|$ is a dominant rational map defined by taking a general flag $V \subseteq H^0(\mathcal{E})$ to the degeneracy locus of the induced morphism $\mathcal{E}^{\vee} \rightarrow V^{\vee} \otimes \mathcal{O}_{X}.$
\end{prop}

\begin{proof}
By Corollary \ref{existssimple}, we have that $\textnormal{Spl}_{X}(2,3H,14)$ is nonempty, and Theorem \ref{simpsmooth} implies that each of its irreducible components is smooth of dimension 14.  We fix an irreducible component $\textnormal{Spl}_{X}^{\circ}(2,3H,14)$ containing the isomorphism class of one of the bundles constructed in Proposition \ref{existssimple}.  Since being weakly Ulrich is an open condition by Proposition \ref{opencond}, the set
\begin{equation}
\label{openweakly}
\mathcal{U}:=\{\mathcal{E} \in \textnormal{Spl}_{X}^{\circ}(2,3H,14) : \mathcal{E} \textnormal{ is weakly Ulrich } \}
\end{equation}
is nonempty and Zariski-open.  Moreover, the fact that weakly Ulrich bundles have no higher cohomology implies (via Riemann-Roch) that $h^0(\mathcal{E})=8$ for all $\mathcal{E} \in \mathcal{U}.$  Consequently there exists a Grassmann bundle  $p: \mathbb{G}_{\mathcal{U}} \rightarrow \mathcal{U}$ whose fiber over $\mathcal{E} \in \mathcal{U}$ is $G(2,H^0(\mathcal{E})) \cong G(2,8)$.  Since $G(2,8)$ is 12-dimensional, the total space $\mathbb{G}_{\mathcal{U}}$ is 26-dimensional.

The dimension of the locus on which a vector bundle fails to be globally generated is upper-semicontinuous in smooth families, so Proposition \ref{existssimple} implies that the general vector bundle $\mathcal{E} \in \mathcal{U}$ is globally generated in codimension 1.  It follows that we have a well-defined rational map $h: \mathbb{G}_{\mathcal{U}} \dashrightarrow |\mathcal{O}_X(3)|$ which takes the general flag $V \subseteq H^0(\mathcal{E})$ to the degeneracy locus of the induced morphism $\mathcal{E}^{\vee} \rightarrow V^{\vee} \otimes \mathcal{O}_{X}.$

We now show that $h$ is dominant.  For general $C$ in the image of $h,$ the fiber $h^{-1}(C)$ is the set of all flags $V \subseteq H^0(\mathcal{E})$ for which the degeneracy locus of $\mathcal{E}^{\vee} \rightarrow V^{\vee} \otimes \mathcal{O}_{X}$ is the curve $C;$ the irreducibility of $\mathbb{G}_{\mathcal{U}}$ implies that $h^{-1}(C)$ is irreducible as well.  It can be seen from (\ref{elmod}) that the elements of $h^{-1}(C)$ are in bijective correspondence with a Zariski-open subset of an irreducible component $\mathcal{W}_{C}$ of $W^1_{14}(C)$ parametrizing basepoint-free degree-14 line bundles with 2 global sections.  Indeed, a general $\mathcal{L} \in \mathcal{W}_C$ gives rise to a sequence of the form (\ref{elmod}), and for each flag $V \subseteq H^0(\mathcal{E})$ in $h^{-1}(C),$ the cokernel of $\mathcal{E}^{\vee} \rightarrow V^{\vee} \otimes \mathcal{O}_{X}$ is an element of $\mathcal{W}_C.$ Since $|\mathcal{O}_X(3)|$ is 19-dimensional, we have that $h^{-1}(C)$ is of dimension at least 7.  By Proposition \ref{lbexist}, we have that $h$ has a 7-dimensional fiber, so the general fiber of $h$ is 7-dimensional; therefore $h$ is dominant.
\end{proof}

\subsection{Proof of Theorem \ref{allquart}}

Before stating the next result, it is worth reminding the reader that $\mathcal{E} \in \mathcal{U}$ need not be globally generated.

\begin{prop}
\label{weakserre}
Let $\mathcal{E}$ be a general element of $\mathcal{U}.$  Then there exists an exact sequence of the form
\begin{equation}
\label{serre1}
0 \rightarrow \mathcal{O}_X \rightarrow \mathcal{E} \rightarrow \mathcal{I}_{Z|X}(3) \rightarrow 0
\end{equation}
where $Z$ is a length-14 l.c.i. subscheme of $X.$
\end{prop}

\begin{proof}
It follows from Proposition \ref{simplemoduli} that the general element $\mathcal{E}$ of $\mathcal{U}$ fits into an exact sequence of the form (\ref{dualtwist1}) for some smooth curve $C \in |\mathcal{O}_X(3)|$ and some basepoint-free line bundle $\mathcal{L}$ on $C$ of degree 14.  Let $Z$ be an element of the linear system $|\mathcal{L}|$ on $C.$   Since $Z$ is a subscheme of a smooth curve in $X,$ it is l.c.i., and we have an exact sequence
\begin{equation}
0 \rightarrow \mathcal{O}_X(-3) \rightarrow \mathcal{I}_{Z|X} \rightarrow \mathcal{L}^{-1} \rightarrow 0
\end{equation}
Let $\sigma:H^{0}(\mathcal{O}_{C}) \rightarrow H^{0}(\mathcal{O}_{C}(Z)) \cong H^0(\mathcal{L})$ be a global section of $\mathcal{L}$ which cuts out $Z$ in $C,$ and define
\begin{equation}
K:=\ker{(\sigma^{\vee}:H^{0}(\mathcal{L})^{\vee} \rightarrow H^{0}(\mathcal{O}_{C})^{\vee})}
\end{equation}
We then have the following commutative diagram with exact rows and columns

\medskip
\begin{center}
\begin{tiny}
$\xymatrix{
            &  &  &0 \ar[d]  & \\
            &0 \ar[d]  &  &H^{0}(\mathcal{O}_{C})^{\vee} \otimes \mathcal{O}_{X} \ar[d] & \\
0 \ar[r] &{K \otimes \mathcal{O}_{X} \cong \mathcal{O}_{X}} \ar[r] \ar[d] &{\mathcal{E}} \ar[r] \ar@{=}[d] &\mathcal{I}_{Z|X}(3) \ar[r] \ar[d] &0 \\
0 \ar[r] &H^{0}(\mathcal{L})^{\vee} \otimes \mathcal{O}_{X} \ar[r] \ar[d] &{\mathcal{E}} \ar[r] &\mathcal{L}^{-1}(3) \ar[r] \ar[d]&0 \\
            &H^{0}(\mathcal{O}_{C})^{\vee} \otimes \mathcal{O}_{X} \ar[d] & &0 & \\
            &0 & & & \\
}$
\end{tiny}
\end{center}
\medskip
whose top row is the desired exact sequence.
\end{proof}

\begin{prop}
\label{quadriculrich}
Let $\mathcal{E} \in \mathcal{U}$ be obtained by an extension
\begin{equation}
\label{serrext0}
0 \rightarrow \mathcal{O}_X \rightarrow \mathcal{E} \rightarrow \mathcal{I}_{Z|X}(3) \rightarrow 0
\end{equation}
where $Z \in X^{[14]}.$  If $H^0(\mathcal{I}_{Z|X}(2))=0$ (i.e. $Z$ does not lie on a quadric in $\mathbb{P}^{3}$) then $\mathcal{E}$ is Ulrich.
\end{prop}

\begin{proof}
Since $\mathcal{E}$ has the Chern classes expected of an Ulrich bundle with $c_1=3H,$ it also has the Hilbert polynomial expected of such an Ulrich bundle, so by Proposition \ref{ulrichacm}  it is enough to check that $\mathcal{E}$ is ACM.  Given that $\mathcal{E}$ is weakly Ulrich, this amounts to showing that $h^1(\mathcal{E}(-1))=h^1(\mathcal{E}(-2))=0.$  The isomorphisms (\ref{serreisoms}) imply that we need only verify the vanishing of $H^1(\mathcal{E}(-1)).$

Twisting (\ref{serrext0}) by $-1$ and taking cohomology, we have the exact sequence
\begin{equation}
0 \rightarrow H^1(\mathcal{E}(-1)) \rightarrow H^1(\mathcal{I}_{Z|X}(2)) \rightarrow H^2(\mathcal{O}_X(-1)) \rightarrow H^2(\mathcal{E}(-1))
\end{equation}
The fact that $\mathcal{E}$ is weakly Ulrich implies that $H^2(\mathcal{E}(-1))=0,$ so the coboundary map $H^1(\mathcal{I}_{Z|X}(2)) \rightarrow H^2(\mathcal{O}_X(-1))$ is surjective.  We have from Serre duality that $h^2(\mathcal{O}_X(-1))=h^0(\mathcal{O}_X(1))=4.$  Moreover, our assumption on $Z$ combined with the exact sequence
\begin{equation}
0 \rightarrow \mathcal{I}_{Z|X}(2) \rightarrow \mathcal{O}_{X}(2) \rightarrow \mathcal{O}_Z(2) \rightarrow 0
\end{equation}
implies that $h^0(\mathcal{I}_{Z|X}(2))=0$ and $h^2(\mathcal{I}_{Z|X}(2)) = h^1(\mathcal{O}_Z(2)) = 0.$  Therefore
\begin{equation}
-h^1(\mathcal{I}_{Z|X}(2)) = \chi(\mathcal{I}_{Z|X}(2)) = \chi(\mathcal{O}_X(2)) - \chi(\mathcal{O}_Z(2)) = 10 - 14 = -4
\end{equation}
All these calculations imply that $h^1(\mathcal{I}_{Z|X}(2))=h^2(\mathcal{O}_{X}(-1))=4.$  It follows that our surjective coboundary map
$H^1(\mathcal{I}_{Z|X}(2)) \rightarrow H^2(\mathcal{O}_X(-1))$ is an isomorphism, e.g. its kernel $H^1(\mathcal{E}(-1))$ is zero.
\end{proof}

It is now clear that Theorem \ref{allquart} will be proved once we show the existence of a weakly Ulrich bundle $\mathcal{E} \in \mathcal{U}$ and a length-14 subscheme $Z$ of $X$ satisfying the hypotheses of Proposition \ref{quadriculrich}; this is the goal of the remaining results in this section.  First, we need a definition.

\begin{defn}
\label{cbdef}
Let $Y$ be a smooth projective surface and let $Z$ be a zero-dimensional l.c.i. subscheme of $Y$ having length $\ell \geq 1.$  If $\mathcal{G}$ is a coherent sheaf on $Y,$ then $Z$ \textit{is Cayley-Bacharach w.r.t. }$\mathcal{G}$ if for any subscheme $Z'$ of $Z$ having length ${\ell}-1,$ the natural inclusion $H^0(\mathcal{I}_{Z|Y} \otimes \mathcal{G}) \hookrightarrow H^0(\mathcal{I}_{Z'|Y} \otimes \mathcal{G})$ is an isomorphism.
\end{defn}

The present importance of the Cayley-Bacharach property stems from the following fundamental result, which is a special case of Theorem 5.1.1 in \cite{HL}.

\begin{thm}
Let $Y$ be a smooth projective surface, let $Z$ be a zero-dimensional l.c.i. subscheme of $Y,$ and let $\mathcal{M}$ be a line bundle on $Y.$  Then there exists an extension of the form
\begin{equation}
\label{genserre}
0 \rightarrow \mathcal{O}_Y \rightarrow \mathcal{F} \rightarrow \mathcal{I}_{Z|Y} \otimes \mathcal{M} \rightarrow 0
\end{equation}
with $\mathcal{F}$ locally free if and only if $Z$ is Cayley-Bacharach w.r.t. $\omega_{Y} \otimes \mathcal{M}.$  \hfill \qedsymbol
\end{thm}

It follows immediately from this theorem that any $Z \in X^{[14]}$ appearing in (\ref{serrext0}) must be Cayley-Bacharach w.r.t. $\mathcal{O}_X(3).$  Our plan for showing that the hypotheses of Proposition \ref{quadriculrich} are satisfied is to first prove that the locus in $X^{[14]}$ parametrizing subschemes of $X$ which lie on a quadric and are also Cayley-Bacharach w.r.t $\mathcal{O}_X(3)$ has dimension at most 20 (Proposition \ref{quadriccb}), and then prove that the locus parametrizing the subschemes $Z$ appearing in  (\ref{serrext0}) is 21-dimensional (Proposition \ref{bigdim}).

We begin with a study of the length-14 subschemes of $X$ which are contained in smooth members of the linear system $|\mathcal{O}_X(2)|.$

\begin{prop}
\label{cbprop}
Let $X \subseteq \mathbb{P}^{3}$ be a smooth quartic surface, and let $C \in |\mathcal{O}_X(2)|$ be a smooth curve.  Then
\begin{equation}
\mathfrak{CB}(C):=\{Z \in C^{(14)} : Z \textnormal{ is Cayley-Bacharach w.r.t. }\mathcal{O}_C(3)\}
\end{equation}
is at most 11-dimensional.
\end{prop}

\begin{proof}
By Definition \ref{cbdef}, we have that
\begin{equation}
\mathfrak{CB}(C)=\{Z \in C^{(14)} : {\forall} p \in \textnormal{supp}(Z) \hskip3pt p \textnormal { is a base point of }|\mathcal{O}_C(3H-Z+p)|\}
\end{equation}
Let $\sigma: C \times C^{(13)} \rightarrow C^{(14)}$ be the addition morphism defined by $\sigma(p,Z') \rightarrow p+Z'.$  Then the locus $\mathfrak{CB}(C)$ is contained in the image under $\sigma$ of the locus
\begin{equation}
\widetilde{\mathfrak{CB}(C)}:=\{(p,Z') \in C \times C^{(13)} : p \textnormal{ is a base point of }\mathcal{O}_{C}(3H-Z')\}
\end{equation}
Given that $\sigma$ is a finite morphism, it suffices to show that $\widetilde{\mathfrak{CB}(C)}$ is at most 11-dimensional.  Observe that the image $\widetilde{\mathfrak{CB}(C)}$ under the projection map $\textnormal{pr}_2 : C \times C^{(13)} \rightarrow C^{(13)}$ is the locus
\begin{equation}
\mathfrak{Bpt}(C):=\{Z' \in C^{(13)} : \mathcal{O}_C(3H-Z') \textnormal{ has a base point }\}
\end{equation}
Any linear series on a curve has finitely many base points, so the restriction of $\textnormal{pr}_2$ to $\widetilde{\mathfrak{CB}(C)}$ is a finite map onto $\mathfrak{Bpt}(C).$  Consequently we are reduced to showing that $\mathfrak{Bpt}(C)$ is at most 11-dimensional.

Since $(2H)^2=16,$ it follows from the adjunction formula that $C$ is a curve of genus 9.  Furthermore, the line bundle $\mathcal{O}_C(3H-Z')$ is of degree 11 for all $Z' \in C^{(13)},$ so Riemann-Roch tells us that $\dim{|\mathcal{O}_C(3H-Z')|} \geq 2$ for all $Z' \in C^{(13)}.$  Consider the commutative diagram
$$\xymatrix{
C^{(13)} \ar[r]^{a} \ar[dr]_{a'} &\textnormal{Pic}^{13}(C) \ar[d]_{\cong}^{\gamma} \\
 &\textnormal{Pic}^{11}(C)
}$$
where $a$ is defined by $Z' \mapsto \mathcal{O}_{C}(Z')$ and $\gamma$ is defined by $\mathcal{M} \mapsto \mathcal{M}^{-1}(3);$ it is straightforward to verify that $a'(C^{(13)}) \subseteq W^{2}_{11}(C)$ and that $\mathfrak{Bpt}(C)=(a')^{-1}(\mathcal{V}),$ where $\mathcal{V}$ is the subvariety of $W^{2}_{11}(C)$ parametrizing line bundles with a base point.

We claim that $6 \leq \dim{\mathcal{V}} \leq 7.$  Note that $\mathcal{V}$ is the image of the morphism $\rho: C \times W^2_{10}(C) \rightarrow W^{2}_{11}(C)$ defined by $(p,\mathcal{M}') \mapsto \mathcal{M}'(p).$  By Serre duality and Riemann-Roch, we have that $W^2_{10}(C) \cong W_{6}(C),$ so $C \times W^2_{10}(C)$ is 7-dimensional; this implies that $\dim(\mathcal{V}) \leq 7.$  To see that $\dim{\mathcal{V}} \geq 6,$ observe that for each $\mathcal{N} \in \mathcal{V},$ any pair $(p,\mathcal{M}') \in {\rho}^{-1}(\mathcal{N})$ must satisfy $\mathcal{M}' \cong \mathcal{N}(-p),$ e.g. $\rho^{-1}(\mathcal{N})$ is contained in the 1-dimensional locus $\{(p,\mathcal{N}(-p)) \in C \times \textnormal{Pic}^{10}(C) : p \in C\}.$
This proves the claim.

Thanks to the fact that $\dim{\mathcal{V}} \leq 7,$ we need only check that for general $\mathcal{N} \in \mathcal{V},$ the fiber $(a')^{-1}(\mathcal{N})$ is 4-dimensional.  Since $\gamma$ is an isomorphism, this amounts to checking that $a^{-1}(\gamma^{-1}(\mathcal{N}))$ is 4-dimensional.  Riemann-Roch implies that $\dim|\mathcal{M}| \geq 4$ for all $\mathcal{M} \in \textnormal{Pic}^{13}(C).$  Another application of Serre duality combined with Riemann-Roch shows that $W^{5}_{13}(C) \cong W_{3}(C),$ so we can have $\dim{|\mathcal{M}|} \geq 5$ only on a 3-dimensional subvariety of $\textnormal{Pic}^{13}(C).$  We have just seen that $\dim{\mathcal{V}} \geq 6,$ so for a general element $\mathcal{M} \in {\gamma}^{-1}(\mathcal{V}),$ we have that $a^{-1}(\gamma^{-1}(\mathcal{M}))$ is of dimension 4.  This concludes the proof.
\end{proof}

\begin{lem}
\label{ctox}
Assume the hypotheses of Proposition \ref{cbprop}.  Then $Z$ is Cayley-Bacharach w.r.t. $\mathcal{O}_C(3)$ if and only if it is Cayley-Bacharach w.r.t. $\mathcal{O}_X(3).$
\end{lem}

\begin{proof}
Let $Z$ be a length-14 subscheme of $C,$ and let $Z'$ be a length-13 subscheme of $Z.$  We then have the following commutative diagram with exact rows and columns:
$$\xymatrix{
        &                                      &0 \ar[d]                         &0 \ar[d]                    &\\
0 \ar[r] &\mathcal{O}_{X}(-2) \ar[r] \ar@{=}[d]     &\mathcal{I}_{Z|X}  \ar[r] \ar[d] &\mathcal{O}_{C}(-Z) \ar[r] \ar[d]  &0 \\
0 \ar[r] &\mathcal{O}_{X}(-2) \ar[r]            &\mathcal{I}_{Z'|X} \ar[r]        &\mathcal{O}_{C}(-Z') \ar[r] &0
}$$
Twisting by 3 and taking cohomology yields the diagram
$$\xymatrix{
 &                                      &0 \ar[d]                         &0 \ar[d]                    &\\
0 \ar[r] &H^0(\mathcal{O}_X(1)) \ar[r] \ar@{=}[d] &H^0(\mathcal{I}_{Z|X}(3)) \ar[r] \ar[d] &H^0(\mathcal{O}_C(3H-Z)) \ar[r] \ar[d] &0 \\
0 \ar[r] &H^0(\mathcal{O}_X(1)) \ar[r] &H^0(\mathcal{I}_{Z'|X}(3)) \ar[r] &H^0(\mathcal{O}_C(3H-Z')) \ar[r] &0
}$$
It follows that $H^0(\mathcal{I}_{Z|X}(3)) \cong H^0(\mathcal{I}_{Z'|X}(3))$ if and only if $H^0(\mathcal{O}_C(3H-Z)) \cong H^0(\mathcal{O}_C(3H-Z')),$ which is what we wanted to show.
\end{proof}

\begin{prop}
\label{quadriccb}
Let $X$ be a smooth quartic surface in $\mathbb{P}^{3},$ and define
$$
\mathfrak{Q}_{CB}:=\{Z \in X^{[14]} : H^0(\mathcal{I}_{Z|X}(2))\neq 0 \textnormal{ and }Z \textnormal{ is Cayley-Bacharach w.r.t. }\mathcal{O}_X(3)\}$$
Then the dimension of $\mathfrak{Q}_{CB}$ is at most 20.
\end{prop}

\begin{proof}
Let $\mathfrak{U} \subseteq |\mathcal{O}_{X}(2)|$ be the Zariski-open subset parametrizing smooth curves, and consider the incidence variety
\begin{equation}
\mathcal{C}^{(14)}_{\mathfrak{U}}:=\{(Z,C) \in X^{[14]} \times \mathfrak{U} : Z \textnormal{ is a subscheme of }C\}
\end{equation}
Since $\mathfrak{U}$ is 9-dimensional, we have from Proposition \ref{cbprop} and Lemma \ref{ctox} that the locus
\begin{equation}
\mathfrak{CB}_{\mathfrak{U}}:=\{(Z,C) \in \mathcal{C}^{(14)}_{\mathfrak{U}} : Z\textnormal{ satisfies Cayley-Bacharach w.r.t. }\mathcal{O}_X(3)\}
\end{equation}
is at most 20-dimensional.  The restriction of the projection map $\textnormal{pr}_{1} : \mathcal{C}^{(14)}_{\mathfrak{U}} \rightarrow X^{[14]}$ to $\mathfrak{CB}_{\mathfrak{U}}$ dominates $\mathfrak{Q}_{CB},$ so the latter is at most 20-dimensional as well.
\end{proof}

We now return to the Grassmann bundle $\mathbb{G}_{\mathcal{U}}$ introduced in Proposition \ref{simplemoduli}.  Recall from the proof of the latter that the general element of $\mathbb{G}_{\mathcal{U}}$ may be described as a pair $(C,\mathcal{L})$ where $\mathcal{L}$ is an element of $W^1_{14}(C)$ satisfying properties (i) through (iv) in the statement of Proposition \ref{weaklyequiv}.  The next result concludes the proof of Theorem \ref{allquart}.

\begin{prop}
\label{bigdim}
Define the incidence scheme
$$\widehat{\mathbb{G}_{\mathcal{U}}}:=\{(Z,(C,\mathcal{L})) \in X^{[14]} \times \mathbb{G}_{\mathcal{U}} : Z \textnormal { is a subscheme of }C\textnormal{ and }\mathcal{O}_C(Z) \cong \mathcal{L}\}$$
Consider the natural maps $\mathfrak{p}_{1}: \widehat{\mathbb{G}_{\mathcal{U}}} \rightarrow X^{[14]}$ and $\mathfrak{p}_{2}: \widehat{\mathbb{G}_{\mathcal{U}}} \rightarrow \mathbb{G}_{\mathcal{U}}.$  We have the following:
\begin{itemize}
\item[(i)]{The locus in $X^{[14]}$ parametrizing the length-14 subschemes $Z$ which appear in (\ref{serrext0}) is birational to $\textnormal{Im}(\mathfrak{p}_{1}).$ }
\item[(ii)]{$\textnormal{Im}(\mathfrak{p}_1)$ is 21-dimensional.}
\end{itemize}
\end{prop}

\begin{proof}
(i) follows from Proposition \ref{weakserre} and the surjective morphism $p \circ \mathfrak{p}_2 : \widehat{\mathbb{G}_{\mathcal{U}}} \rightarrow \mathcal{U}.$  Turning to (ii), we observe that for general $Z \in \textnormal{Im}(\mathfrak{p}_1),$ we have that $\mathfrak{p}_{1}^{-1}(Z)=|\mathcal{I}_{Z|X}(3)|$ and for general $(C,\mathcal{L}) \in \textnormal{Im}(\mathfrak{p}_2),$ we have $\mathfrak{p}_{2}^{-1}((C,\mathcal{L}))=|\mathcal{L}| \cong \mathbb{P}^{1}.$  We have already seen in the proof of Proposition \ref{simplemoduli} that $\mathbb{G}_{\mathcal{U}}$ is irreducible and 26-dimensional.  Since $\mathfrak{p}_2$ is dominant and its general fiber is 1-dimensional, we have that $\widehat{\mathbb{G}_{\mathcal{U}}}$ is irreducible and 27-dimensional.

Given that (i) holds, taking cohomology in (\ref{serrext0}) and using the fact that any $\mathcal{E} \in \mathcal{U}$ has 8 global sections shows that $|\mathcal{I}_{Z|X}(3)|$ is 6-dimensional for general $Z \in \textnormal{Im}(\mathfrak{p}_1),$ i.e. that the general fiber of $\mathfrak{p}_1$ is 6-dimensional.  This proves (ii).
\end{proof}



\subsection{Stability Of Simple Ulrich Bundles}
As stated in the introduction, we would like to know when the simple Ulrich bundles we have constructed on our smooth quartic surface are stable.  In the case where ${\rm Pic}(X) \cong \mathbb{Z}H$ we have an affirmative answer; the next result gives a partial answer for the case of higher Picard number.    

\begin{prop}
\label{simpstab}
Let $X \subseteq \mathbb{P}^{3}$ be a smooth quartic surface whose Picard number is at least 3 and which contains finitely many smooth rational curves, and let $\mathcal{E}$ be a general simple rank-2 vector bundle on $X$ with $c_1(\mathcal{E})=3H$ and $c_2(\mathcal{E})=14.$  Then $\mathcal{E}$ is stable.
\end{prop}

\begin{proof}
By Theorem 2 in \cite{Qin1} and Theorem A in \cite{Qin2}, it suffices to show that for any divisor class $F \in \textnormal{Pic}(X)$, the following statements hold:
\begin{itemize}
\item[(i)]{$2F$ is not numerically equivalent to $3H.$}
\item[(ii)]{If $2F-3H$ has positive intersection with every ample divisor on $X,$ then $2F-3H$ is the class of an effective divisor on $X.$}
\end{itemize}
If $2F$ is numerically equivalent to $3H,$ then $F^2=9,$ which contradicts the fact that the self-intersection of any divisor on a K3 surface is even; therefore (i) is true.  If $2F-3H$ has positive intersection with every ample divisor on $X_f,$ then $2F-3H$ is in the closure of the effective cone of $X$ (e.g. Section 1.4C of \cite{Laz1}).  Our hypothesis on $X$ implies that the effective cone of $X$ is polyhedral (e.g. Theorem 1 of \cite{Kov}) and therefore closed; this concludes the proof of (ii).
\end{proof}

\begin{rem}
The theorems of Qin that we have just used refer to slope-stability, rather than stability.  However, slope-stability implies stability, and by Theorem 2.9(c) of \cite{CH} the two notions coincide for Ulrich bundles on any smooth projective variety.  
\end{rem}

\section{Constructing Irreducible Representations of The Clifford Algebra}
\label{finalcliff}
\subsection{The geometry of Clifford quartic surfaces}

For the rest of the paper, $X_{f}$ will denote the smooth quartic surface defined by the equation $w^4=f$ for a general nondegenerate ternary quartic form $f=f(x_1,x_2,x_3);$ we refer to $X_f$ as a \textit{Clifford quartic}.

The starting point for our study of the geometry of $X_f$ is the observation that $\pi$ factors as a composition of two double covers.  More precisely, we have that $\pi = \pi_1 \circ \pi_2,$ where $\pi_1:S \to \mathbb{P}^{2}$ is the double cover of $\mathbb{P}^2$ branched over $Q,$ and $\pi_2:X_f \to S$ is the double cover of $S$ branched over $B:=\pi_2^{-1}(Q).$

Since the canonical bundle $\omega_{S}$ of $S$ is isomorphic to $\pi_1^{\ast}(\omega_{\mathbb{P}^{2}}(2)) \cong \pi_1^{\ast}\mathcal{O}_{\mathbb{P}^{2}}(-1),$ it follows that $S$ is a del Pezzo surface of degree 2.  As is well-known, $S$ can also be described as the blow-up  $\phi:S \to \mathbb{P}^{2}$ of the projective plane at seven points in general position.  Consequently, its Picard group $\textnormal{Pic}(S)$ is the free abelian group of rank 8 generated by the classes $e_0,e_1,\ldots,e_7,$  where $e_0$ is the class of $\phi^{\ast}\mathcal{O}_{\mathbb{P}^{2}}(1)$ and $e_1, \cdots ,e_7$ are the classes of the exceptional divisors associated to $\phi.$  The matrix of the intersection form on $\textnormal{Pic}(S)$ with respect to the ordered basis $\{e_0, \cdots ,e_7\}$ is $\textnormal{diag}(1,-1, \cdots ,-1).$  The following result says that the pullbacks of these classes via $\pi_2$ generate all of $\textnormal{Pic}(X_f);$ see Proposition 1.4 of \cite{Art} for the proof (and for a slightly more general statement).

\begin{prop}
\label{picardlattice}
For $i=0, \cdots ,7,$ define $\tilde{e}_i:=\pi_2^{\ast}e_i.$  Then the classes $\widetilde{e}_0, \cdots ,\widetilde{e}_{7}$ form a basis for the Picard lattice $\textnormal{Pic}(X_f)$, whose intersection form is given by $\mbox{diag}(2,-2,\ldots,-2)$.  In particular, the intersection form on $\textnormal{Pic}(X)$ is $2\mathbb{Z}-$valued.  \hfill \qedsymbol
\end{prop}

There are 56 curves on $S$ with self-intersection -1, and their classes are given by $e_i$ for $i=1,\ldots,7$; $e_0-e_i-e_j$ for $1 \leq i,j \leq 7$ ($i \neq j$); $2e_0-e_{i_1}- \cdots -e_{i_5}$ (where $\{i_1, \cdots ,i_5\} \subseteq \{1, \cdots ,7\}$) and $3e_0-\sum_{i \neq j}e_i-2e_j$.  The set consisting of these 56 curve classes admits a free $\mathbb{Z}/2-$action via the deck involution of $\pi_{1},$ and the 28 orbits of this action are in bijective correspondence with the 28 bitangent lines of the quartic plane curve $Q.$

\subsection{Existence of stable Ulrich bundles on $X_f$}
\label{stabexist}
Propositions \ref{simpstab} and \ref{picardlattice}, together with the fact that there are exactly 56 smooth rational curves on $X_f,$ imply the following result, which in turn implies Theorem \ref{main}.

\begin{prop}
\label{stable}
There is a Zariski-open subset $\mathcal{U'}$ of the set $\mathcal{U}$ defined in (\ref{openweakly}) whose members are stable Ulrich bundles of rank 2 on $X_f.$ \hfill \qedsymbol
\end{prop}

This subsection contains an alternate proof of Proposition \ref{stable} that yields slightly sharper information, namely that \textit{all} the Ulrich bundles on $X_f$ guaranteed by Theorem \ref{allquart} are stable.  We begin with a well-known lemma, whose proof is included for completeness.
\begin{lem}
\label{K3eff}
Let $D$ be a divisor class on a K3 surface $X$ such that $D^2 > 0$ and $D$ has positive intersection with some ample divisor $H$ on $X.$  Then $D$ is the class of an effective divisor on $X.$
\end{lem}

\begin{proof}
By the Riemann-Roch theorem and Serre Duality, we have that
\begin{equation}
h^0(\mathcal{O}_{X}(D))+h^0(\mathcal{O}_{X}(-D)) \geq 2 + \frac{D^2}{2}
\end{equation}
Since $D^2 > 0$ by hypothesis, it follows that exactly one of $D$ and $-D$ is an effective divisor class.  Since $-D$ is assumed to have negative intersection with an ample divisor on $X,$ it cannot be effective; this concludes the proof.
\end{proof}

\begin{prop}
\label{curvenef}
There does not exist an effective divisor $D$ on $X$ satisfying $D \cdot H=6$ and $D^2=4.$
\end{prop}

\begin{proof}
First we will show that any effective divisor class $D$ with $D \cdot H=6$ and $D^2=4$ must have nonnegative intersection with any of the 56 conics on $X,$ and afterwards we will show that this yields a contradiction.

We assume that $D$ is linearly equivalent to a reducible curve of the form $\tilde{e}+D'$ for a conic $\tilde{e}$ on $X$ and a curve $D'$ on $X,$ since the statement of the lemma is immediate otherwise.  First, note that $D' $ is irreducible.  Indeed, we have that $D' \cdot H =4,$ and if $D'$ is reducible then it must be the sum of two conics $\tilde{e}'$ and $\tilde{e}''$ since $X$ does not contain any lines.  Simplifying the equation $D^2=(\tilde{e}+\tilde{e}'+\tilde{e}'')^{2}=4$ yields
\begin{equation}
\tilde{e} \cdot \tilde{e}' + \tilde{e} \cdot \tilde{e}'' + \tilde{e}' \cdot \tilde{e}'' = 5
\end{equation}
which is impossible since the intersection form on $X$ is even.  Given that $D'$ and $\tilde{e}$ are distinct irreducible curves (their intersections with $H$ differ), it follows that $D' \cdot \tilde{e} \geq 0.$  Moreover, if $\tilde{e}'''$ is any of the 55 conics on $X$ which is not equal to $\tilde{e},$ then $D \cdot \tilde{e}''' = (\tilde{e}+D') \cdot \tilde{e}''' \geq 0.$

It remains to check that $D \cdot \tilde{e} \geq 0,$ or equivalently that $D' \cdot \tilde{e} \geq 2.$  All that is required is to rule out the possibility that $D' \cdot \tilde{e}=0.$

Note that $D'$ is not linearly equivalent to $H;$ this is because the arithmetic genus of $\tilde{e}+D'$ is 3, whereas the arithmetic genus of $\tilde{e}+H$ is 4.  The Hodge Index Theorem then implies that $(D'-H)^2<0,$ so that $(D')^2 \leq 2.$  If $D' \cdot \tilde{e}=0,$ then $D^2=(D')^2+\tilde{e}^2 \leq 0,$ which is absurd.

Write $D=a\widetilde{e}_0-\sum b_i \widetilde{e}_i$.  By hypothesis, we have
\begin{align*}
 D.H &= 6a-\sum 2b_i &= 6 \\
 D^2 &= 2a^2-2\sum b_i^2 &= 4.
\end{align*}

We have just shown that $D$ has nonnegative intersection with any conic on $X,$ so we also have the following inequalities
\begin{align*}
 D.(\widetilde{e}_i) &= 2b_i &\geq 0 \\
 D.(\widetilde{e}_0-\widetilde{e}_i-\widetilde{e}_j) &= 2a-2b_i-2b_j &\geq 0 \\
 D.(2\widetilde{e}_0-e_{i_1}- \cdots -e_{i_5}) &= 4a-2\sum_{t=1}^{5}b_{i_t} &\geq 0 \\
 D.(3\widetilde{e}_0-\sum_{i \neq j}\widetilde{e}_i-2\widetilde{e}_j) &= 6a-2\sum_{i \neq j}b_i-4b_j &\geq 0
\end{align*}
The first formula gives $b_i \geq 0$. Combining the last formula with $3a-\sum b_i = 3$ gives $b_i \leq 3$. Hence it is possible to check solutions using a computer program.\footnote{Available from the first author upon request.}  Doing this shows that our system of inequalities is inconsistent, e.g. that $D$ cannot exist.
\end{proof}

\begin{prop}
\label{halfstab}
Any semistable bundle of rank 2 on $X_f$ with $c_1=3H$ and $c_2=14$ is stable.  In particular, any Ulrich bundle of rank 2 on $X_f$ with $c_1=3H$ and $c_2=14$ is stable.
\end{prop}

\begin{proof}
Let $\mathcal{E}$ be a bundle satisfying the hypotheses.  A Hilbert polynomial calculation shows that any line bundle $\mathcal{M}$ on $X_f$ which destabilizes $\mathcal{E}$ must satisfy $c_1(\mathcal{M}) \cdot H = 6$ and $c_1(\mathcal{M})^2=4.$  However, Lemma \ref{K3eff} and Proposition \ref{curvenef} imply that such a line bundle cannot exist on $X_f.$   
\end{proof}

\begin{cor}
Each irreducible component of the moduli space of rank-2 stable bundles on $X_f$ is a smooth projective variety of dimension 14. \hfill \qedsymbol  
\end{cor} 

\subsection{A simple weakly Ulrich bundle on $X_f$ which is not Ulrich}
\label{badbundles}
In this final subsection, we construct a family of weakly Ulrich bundles on $X_f$ with $c_1=3H$ and $c_2=14$ whose members are simple and globally generated, but not Ulrich.

Let $E$ be a smooth cubic curve in $\mathbb{P}^{2}$ which intersects the branch divisor of $\pi_{1}$ in 12 distinct points.  Then $D:=\pi_{1}^{-1}(E)$ is a smooth irreducible member of the linear system $|\mathcal{O}_{S}(3)|.$  Since ${\pi}_{1}|_{D}$ is a double cover of an elliptic curve which is ramified at 12 points, it follows easily from Riemann-Hurwitz that $D$ is of genus 7.  In what follows, $C$ will always denote ${\pi_{2}}^{-1}(D).$  Note that $C$ is the member of $|\mathcal{O}_{X}(3)|$ obtained by intersecting $X$ with the cone over $E \subseteq \mathbb{P}^2.$

\begin{lem}
\label{branchbundle}
$(\pi_{2}|_{C})_{\ast}\mathcal{O}_{C} \cong \mathcal{O}_{D} \oplus \mathcal{O}_{D}(-1).$
\end{lem}

\begin{proof}
Since the trace map $\textnormal{tr}: (\pi_{2}|_{C})_{\ast}\mathcal{O}_{C} \rightarrow \mathcal{O}_{D}$ induces the splitting $(\pi_{2}|_{C})_{\ast}\mathcal{O}_{C} \cong \mathcal{O}_{D} \oplus \mathcal{N},$ where $\mathcal{N}:=\det((\pi_{2}|_{C})_{\ast}\mathcal{O}_{C}),$ it suffices to show that $\mathcal{N} \cong \mathcal{O}_{D}(-1).$

The branch divisor $B$ of $\pi_{2}|_{C}$ is equal to the ramification divisor of $\pi_{1}|_{D} : D \rightarrow E,$ which is in turn a member of the canonical linear system $|\omega_{D}|=|\mathcal{O}_{D}(2)|.$  We then have that
\begin{equation}
\mathcal{N}^{\otimes 2} \cong \mathcal{O}_{D}(-B) \cong \mathcal{O}_{D}(-2)
\end{equation}
Given that the degree of $B$ is 12, the degree of $\mathcal{N}$ is -6, so that $\mathcal{N}(1)$ has degree 0.  We have from the projection formula that
\begin{equation}
H^{0}(\mathcal{O}_{C}(1)) \cong H^{0}(\mathcal{O}_{D}(1) \otimes (\pi_{2})_{\ast}\mathcal{O}_{C}) \cong H^{0}(\mathcal{O}_{D}(1)) \oplus H^{0}(\mathcal{N}(1)).
\end{equation}
Taking cohomology of both the exact sequence
\begin{equation}
\label{curverest}
0 \rightarrow \mathcal{O}_{X}(t-3) \rightarrow \mathcal{O}_{X}(t) \rightarrow \mathcal{O}_{C}(t) \rightarrow 0
\end{equation} 
and its analogue for $D \subseteq S$ when $t=1,$ we see that $h^{0}(\mathcal{O}_{C}(1))=4$ and $h^{0}(\mathcal{O}_{D}(1))=3,$  so $h^{0}(\mathcal{N}(1))=1.$  In particular, $\mathcal{N}(1)$ is trivial.
\end{proof}

\begin{prop}
\label{genus7}
There exists a line bundle $\mathcal{L}'$ on $D$ of degree 7 which satisfies the following properties:
\begin{itemize}
\item[(i)]{$\mathcal{L}'$ and $\mathcal{L}'^{-1}(3)$ are globally generated.}
\item[(ii)]{$h^{0}(\mathcal{L}')=2.$}
\item[(iii)]{$h^{0}(\mathcal{L}'(-1))=0.$}
\item[(iv)]{$h^1(\mathcal{L}'^{\otimes 2}(-1))=0.$}
\end{itemize}
\end{prop}

\begin{proof}
We will show that the general member of the Brill-Noether locus $W^{1}_{7}(D)$ satisfies conditions (i) through (iv).  Since $W^1_7(D) \cong W_5(D)$ by Riemann-Roch and Serre duality, we have that $W^{1}_{7}(D)$ is 5-dimensional, so it suffices to check that the locus of elements of $W^{1}_{7}(D)$ which violate one of these conditions is at most 4-dimensional.

Since $D$ is bielliptic, the Castelnuovo-Severi inequality (e.g. Exercise C-1 in Section VIII of \cite{ACGH}) implies that $D$ cannot be hyperelliptic.  Consequently Martens' Theorem implies that the dimension of $W^{1}_{6}(D)$ is at most 3.  In particular, the image of the addition map $\sigma : D \times W^{1}_{6}(D) \rightarrow W^{1}_{7}(D)$ defined by $\sigma(p,\mathcal{L}')=\mathcal{L}'(p),$ which parametrizes the elements of $W^{1}_{7}(D)$ that fail to be globally generated, has dimension most 4.

Turning to $\mathcal{L}'^{-1}(3),$ we see that since the latter has degree 11 and that its Serre dual $\omega_{D} \otimes \mathcal{L}'(-3) \cong \mathcal{L}'(-1)$ has degree 1, so that $H^1(\mathcal{L}'^{-1}(3)) \neq 0$ only for $\mathcal{L}'$ in a 1-dimensional locus.  Consequently $\mathcal{L}'^{-1}(3)$ is nonspecial for general $\mathcal{L}' \in W^{1}_{7}(D);$ in particular, $h^{0}(\mathcal{L}'^{-1}(3))=5$ for all such $\mathcal{L}'.$  The locus in $W^1_{7}(D)$ for which $\mathcal{L}'^{-1}(3)$ fails to be globally generated is isomorphic to the image of the map ${\sigma}' : D \times W^{4}_{10}(D) \rightarrow W^4_{11}(D).$  Given that $D$ is nonhyperelliptic, the dimension of $W^{4}_{10}(D)$ is equal to 2, so the image of ${\sigma}'$ has dimension at most 3.  We may conclude that $\mathcal{L}'^{-1}(3)$ is globally generated for general $\mathcal{L}' \in W^1_{7}(D).$

The Brill-Noether locus $W^{2}_{7}(D)$ parametrizes the elements of $W^{1}_{7}(D)$ which violate (ii).  Another application of Martens' Theorem shows that it has dimension at most 2.

Since $\mathcal{O}_{D}(1)$ is of degree 6, we have that $\mathcal{L}'(-1)$ is of degree 1 for all $\mathcal{L}' \in \textnormal{Pic}^{7}(D).$  Given that line bundles of degree 1 with a global section are parametrized by a copy of the curve $D$ according to Abel's Theorem, we have that $h^{0}(\mathcal{L}'(-1))=0$ for general $\mathcal{L}' \in W^{1}_{7}(D),$ so condition (iii) is proved.

Finally, we come to condition (iv).  Since $\mathcal{L}'^{\otimes 2}(-1)$ has degree 8, its Serre dual $\omega_{D} \otimes \mathcal{L}'^{\otimes -2}(1)$ is of degree 4, so by Serre duality the nonvanishing of $H^{1}(\mathcal{L}'^{\otimes 2}(-1))$ is equivalent to the degree-4 line bundle $\omega_{D} \otimes \mathcal{L}'^{\otimes -2}(1)$ having a nonzero global section.  Since $W_{4}(D)$ is 4-dimensional, we have that the general member of $W^{1}_{7}(D)$ satisifies condition (iv).
\end{proof}

\begin{lem}
For all $m \in \mathbb{Z},$ the pullback morphism $\pi_{2}^{\ast}:\textnormal{Pic}^{m}(D) \rightarrow \textnormal{Pic}^{2m}(C)$ is injective.
\end{lem}

\begin{proof}
It is enough to show that if $\mathcal{M}$ is a line bundle on $D$ for which $\pi_{2}^{\ast}\mathcal{M} \cong \mathcal{O}_{C},$ then $\mathcal{M} \cong \mathcal{O}_{D}.$  Given that $2 \cdot c_{1}(\mathcal{M})=c_{1}({\pi_{2}}^{\ast}\mathcal{M})=0,$ we need only check that $h^{0}(\mathcal{M})=1.$  By Lemma \ref{branchbundle}, we have that $H^{0}(\mathcal{O}_{C}) \cong H^{0}(\mathcal{M}) \oplus H^{0}(\mathcal{M}(-1)) \cong H^{0}(\mathcal{M});$ this concludes the proof.
\end{proof}

\begin{prop}
\label{bncomp}
$(\pi_{2}^{\ast})(W^{1}_{7}(D))$ is contained in $W^{1}_{14}(C),$ and there is a unique irreducible component $\mathcal{W}$ of $W^{1}_{14}(C)$ which contains $(\pi_{2})^{\ast}W^{1}_{7}(D)$ and is generically smooth of dimension 7.
\end{prop}

\begin{proof}
Lemma \ref{branchbundle} implies that for all $\mathcal{L}' \in W^{1}_{7}(D),$ we have that $h^{0}(\pi_{2}^{\ast}\mathcal{L}')=h^{0}(\mathcal{L}')+h^{0}(\mathcal{L}'(-1)) \geq 2,$ so the inclusion of $(\pi_{2}^{\ast})(W^{1}_{7}(D))$ in $W^{1}_{14}(C)$ immediately follows.  For the rest of the proof, we fix $\mathcal{L}' \in W^{1}_{7}(D)$ sufficiently general enough to satisfy conditions (i) through (iv) in the statement of Proposition \ref{genus7}.

Let $\mathcal{W}$ be an irreducible component of $W^{1}_{14}(C)$ which contains $(\pi_{2}^{\ast})(W^{1}_{7}(D)).$  According to Proposition 4.2 in Chapter IV of \cite{ACGH}, $\mathcal{W}$ is smooth of (expected) dimension 7 at $\pi_{2}^{\ast}\mathcal{L}$ if and only if the multiplication map $\mu:H^{0}(\pi_{2}^{\ast}\mathcal{L}') \otimes H^{0}(\omega_{C} \otimes \pi_{2}^{\ast}\mathcal{L}'^{-1}) \rightarrow H^{0}(\omega_{C})$ is injective.  Since $\omega_{C} \cong \mathcal{O}_{C}(3),$ the basepoint-free pencil trick implies that the kernel of $\mu$ is isomorphic to $H^{0}(\pi_{2}^{\ast}\mathcal{L}'^{-2}(3)) \cong H^{1}(\pi_{2}^{\ast}\mathcal{L}'^{2})^{\vee}.$  Another application of Lemma \ref{branchbundle} shows that by our assumptions on $\mathcal{L}',$ we have
\begin{equation}
H^1(\pi_{2}^{\ast}\mathcal{L}'^{2}) \cong H^1(\mathcal{L}'^2) \oplus H^1(\mathcal{L}'^{\otimes 2}(-1)) \cong 0
\end{equation}
It follows that the general point of $(\pi_{2}^{\ast})(W^{1}_{7}(D))$ is a smooth point of $W^{1}_{14}(C)$.  This implies that $\mathcal{W}$ is generically smooth and also that $\mathcal{W}$ is the unique irreducible component of $W^{1}_{14}(C)$ which contains $(\pi_{2}^{\ast})(W^{1}_{7}(D)).$
\end{proof}
\begin{prop}
\label{3outof4}
The general member $\mathcal{L}$ of $\mathcal{W}$ satisfies the following:
\begin{itemize}
\item[(i)]{$\mathcal{L}$ and $\mathcal{L}^{-1}(3) \cong \omega_{C} \otimes \mathcal{L}^{-1}$ are globally generated.}
\item[(ii)]{$h^{0}(\mathcal{L})=2.$}
\item[(iii)]{$h^{0}(\mathcal{L}(-1))=0.$}
\end{itemize}
\end{prop}

\begin{proof}
Let $\mathcal{P}$ be the restriction of a degree-14 Poincar\'{e} bundle on $C \times \textnormal{Pic}^{14}(C)$ to $C \times \mathcal{W},$ and let $\mathfrak{p}:C \times \mathcal{W} \rightarrow \mathcal{W}$ be projection.  For each $\mathcal{L} \in \mathcal{W},$ the restriction of the natural evaluation morphism $\textnormal{ev}: \mathfrak{p}^{\ast}\mathfrak{p}_{\ast}\mathcal{P} \rightarrow \mathcal{P}$ to each fiber ${\mathfrak{p}^{-1}}(\mathcal{L}) \cong C$ is the evaluation map $\textnormal{ev}_{\mathcal{L}}:H^{0}(\mathcal{L}) \otimes \mathcal{O}_{C} \rightarrow \mathcal{L},$ whose surjectivity is equivalent to $\mathcal{L}$ being globally generated, so (i) is a Zariski-open condition on $\mathcal{W}.$  Furthermore, we have that $h^{0}(\mathcal{L}) \geq 2$ and $h^{0}(\mathcal{L}(-1)) \geq 0$ for all $\mathcal{L} \in \mathcal{W},$ so (ii) and (iii) are also Zariski-open conditions on $\mathcal{W}$ by semicontinuity.

Propositions \ref{genus7} and \ref{bncomp} imply that the general member of $(\pi_{2}^{\ast})(W^{1}_{7}(D))$ is a smooth point of $\mathcal{W}$ which satisfies conditions (i),(ii), and (iii); this concludes the proof.
\end{proof}
\begin{prop}
\label{ggstable}
Let $\mathcal{E}$ be a rank-2 vector bundle constructed from the general member $\mathcal{L}$ of $\mathcal{W}.$  Then $\mathcal{E}$ is weakly Ulrich, globally generated, and simple.  In particular, the isomorphism class of $\mathcal{E}$ lies in $\mathcal{M}_{H}^{\mathfrak{g}}(2,3H,14).$
\end{prop}
\begin{proof}
Proposition \ref{weaklyequiv} implies that $\mathcal{E}$ is simple and weakly Ulrich.  The exact sequence (\ref{dualtwist1}) implies that $\mathcal{E}$ is globally generated away from the base locus of $\mathcal{L}^{-1}(3).$  Since the latter is a globally generated line bundle on $C$ by (i) of Proposition \ref{3outof4}, $\mathcal{E}$ is globally generated.
\end{proof}

\begin{prop}
Let $\mathcal{E}$ be a rank-2 vector bundle constructed from the general member $\mathcal{L}$ of $\mathcal{W}.$  Then $\mathcal{E}$ is not Ulrich.
\end{prop}

\begin{proof}
It is enough to check that $H^1(\mathcal{E}(-1)) \neq 0.$  By (\ref{cokerisom}), it suffices in turn to check that the multiplication map $H^0(\mathcal{O}_X(1))^2 \rightarrow H^0(\mathcal{L}(1))$ is not surjective.  Since $\mathcal{L}$ is a general member of $\mathcal{W},$ we have that $\mathcal{L} \cong \pi_2^{\ast}\mathcal{L}'$ for general $\mathcal{L}' \in W^1_7(D).$  Lemma \ref{branchbundle} then implies that
\begin{equation}
H^0(\mathcal{L}(1)) \cong H^0(\mathcal{L}'(1)) \oplus H^0(\mathcal{L}')
\end{equation}
Since $h^0(\mathcal{L}'(1))=7$ and $h^0(\mathcal{L}')=2,$ we have that $h^0(\mathcal{L}(1))=9,$ so it follows that the multiplication map cannot be surjective for dimension reasons.
\end{proof}

\begin{rem}
\label{cfweak}
In \cite{CF}, Chiantini and Faenzi show that a general quintic surface in $\mathbb{P}^3$ admits a linear Pfaffian representation by constructing a weakly Ulrich bundle on a specific quintic surface and deforming it to an Ulrich bundle on a nearby quintic surface (Section 6.3 in \textit{loc.~cit.}).  One can recover the Beauville-Schreyer result on general quartic surfaces by applying their deformation to the weakly Ulrich bundle we have just constructed. 
\end{rem}


\begin{thebibliography}{10}

\begin{singlespace}
\bibitem[ACGH]{ACGH}
E. Arbarello, M. Cornalba, P. Griffiths and J. Harris, {\sl Geometry of Algebraic Curves, Vol. 1}, Springer-Verlag, New York (1986)
\end{singlespace}
\begin{singlespace}
\bibitem[AF]{AF}
M. Aprodu and G. Farkas, {\sl Green's Conjecture for Curves on Arbitrary K3 Surfaces}, to appear in Compositio Math.
\end{singlespace}
\begin{singlespace}
\bibitem[Bas]{Bas}
B. Basili, {\sl Indice de Clifford des intersections compl\`{e}tes de l'espace}, Bull. Soc. Math. France \textbf{124} (1996), no. 1, p. 61�95
\end{singlespace}
\begin{singlespace}
\bibitem[Art]{Art}
M. Artebani, {\sl A compactification of $\mathcal{M}_{3}$ via K3 surfaces}, Nagoya Math. J. \textbf{196} (2009), p. 1-26.
\end{singlespace}
\begin{singlespace}
\bibitem[BHS]{BHS}
J. Backelin, J. Herzog, and H. Sanders, {\sl Matrix factorizations of homogeneous polynomials,} in \textit{Proceedings of the 5th National School in Algebra held in Varna, Bulgaria, Sept. 24 � Oct. 4, 1986}, Springer Lecture Notes in Mathematics 1352 (1988), p. 1-33
\end{singlespace}
\begin{singlespace}
\bibitem[Bea]{Bea}
A. Beauville, {\sl Determinantal Hypersurfaces}, Michigan Math. J. \textbf{48} (2000), p.39-64
\end{singlespace}
\begin{singlespace}
\bibitem[BF]{BF}
M. Brambilla and D. Faenzi, {\sl Moduli spaces of rank 2 ACM bundles on prime Fano threefolds}, to appear in Michigan Math. J.
\end{singlespace}
\begin{singlespace}
\bibitem[BHU]{BHU}
J. Brennan, J. Herzog, and B. Ulrich, {\sl Maximally generated Cohen-Macaulay modules}, Math. Scand. \textbf{61} (1987), no. 2, p. 181�203.
\end{singlespace}
\begin{singlespace}
\bibitem[CH]{CH}
M. Casanellas and R. Hartshorne, {\sl Stable Ulrich Bundles}, preprint.
\end{singlespace}
\begin{singlespace}
\bibitem[CF]{CF}
L. Chiantini and D. Faenzi, {\sl Rank 2 arithmetically Cohen-Macaulay bundles on a general quintic surface}, Math. Nachr. \textbf{282} (2009), no. 12, p. 1691�1708.
\end{singlespace}
\begin{singlespace}
\bibitem[CKM]{CKM}
E. Coskun, R. Kulkarni, and Y. Mustopa, {\sl The geometry of Ulrich bundles on del Pezzo surfaces}, in preparation.
\end{singlespace}
\begin{singlespace}
\bibitem[CKM2]{CKM2}
E. Coskun, R. Kulkarni, and Y. Mustopa, {\sl On Representations of Clifford Algebras of Ternary Cubic Forms}, submitted.
\end{singlespace}
\begin{singlespace}
\bibitem[Dem]{Dem}
M. Demazure, {\sl Surfaces de del Pezzo II}, S\'{e}minaire sur les singularit\'{e}s des surfaces (H. Pinkham, M. Demazure, and B. Teissier, eds.), Lecture Notes in Mathematics \textbf{777} (1980), p. 23-35
\end{singlespace}
\begin{singlespace}
\bibitem[ESW]{ESW}
D. Eisenbud, F.-O. Schreyer, and J. Weyman, {\sl Resultants and Chow forms via exterior syzygies}, J. Amer. Math. Soc. \textbf{16} (2003), no. 3, 537-579
\end{singlespace}
\begin{singlespace}
\bibitem[Fog]{Fog}
J. Fogarty, {\sl Algebraic families on an algebraic surface}, Amer. J. Math,
\textbf{90} (1968), 511�521.
\end{singlespace}
\begin{singlespace}
\bibitem[HL]{HL}
D. Huybrechts and M. Lehn, {\sl The Geometry of Moduli Spaces of Sheaves, 2nd ed.}, Cambridge University Press, Cambridge (2010)
\end{singlespace}
\begin{singlespace}
\bibitem[HT]{HT}
D. Haile and S. Tesser, {\sl On Azumaya algebras arising from Clifford algebras}, J. Algebra \textbf{116} (1988), no. 2, 372-384
\end{singlespace}
\begin{singlespace}
\bibitem[IM]{IM}
A. Iliev and D. Markushevich, {\sl Quartic 3-fold: Pfaffians, vector bundles, and half-canonical curves}, Michigan Math. J. \textbf{47} (2000), no. 2, 385�394
\end{singlespace}
\begin{singlespace}
\bibitem[Kon]{Kon}
S. Kond\={o}, {\sl A complex hyperbolic structure for the moduli space of curves of genus three}, J. Reine. Angew. Math. \textbf{525} (2000), p. 219-232
\end{singlespace}
\begin{singlespace}
\bibitem[Laz]{Laz}
R. Lazarsfeld, {\sl Brill-Noether-Petri without degenerations}, J. Differential Geometry \textbf{23} (1986), p. 299-307
\end{singlespace}
\begin{singlespace}
\bibitem[Laz1]{Laz1}
R. Lazarsfeld, {\sl Positivity in Algebraic Geometry, Vol. 1}, Springer-Verlag, New York (2004)
\end{singlespace}
\begin{singlespace}
\bibitem[MP]{MP}
R. Mir\'{o}-Roig and J. Pons-Llopis, {\sl $N-$dimensional Fano Varieties of Wild Representation Type}, preprint, 2010
\end{singlespace}
\begin{singlespace}
\bibitem[Muk]{Muk}
S. Mukai, {\sl Symplectic structure on the moduli space of sheaves on an Abelian or K3 surface,} Invent. Math. \textbf{77} (1984), p. 101-116
\end{singlespace}
\begin{singlespace}
\bibitem[OSS]{OSS}
C. Okonek, M. Schneider, and H. Spindler, {\sl Vector Bundles on Complex Projective Spaces}, Progress in Mathematics \textbf{3}, Birkh\"{a}user (1980)
\end{singlespace}
\begin{singlespace}
\bibitem[Kov]{Kov}
S. Kov\'{a}cs, {\sl The cone of curves of a K3 surface,} Math. Ann. \textbf{300} (1994), no. 4, p. 681-691
\end{singlespace}
\begin{singlespace}
\bibitem[Qin1]{Qin1}
Z. Qin, {\sl Chamber structures of algebraic surfaces with Kodaira dimension zero and 
moduli spaces of stable rank two bundles}, Math. Z. \textbf{207} (1991), p.121-136 
\end{singlespace}
\begin{singlespace}
\bibitem[Qin2]{Qin2}
Z. Qin, {\sl Simple sheaves versus stable sheaves on algebraic surfaces}, Math. Z. \textbf{209} (1992), p. 559-579
\end{singlespace}
\begin{singlespace}
\bibitem[Ulr]{Ulr}
B. Ulrich, {\sl Gorenstein rings and modules with high numbers of generators}, Math. Z. \textbf{188} (1984), no. 1, p. 23�32. 
\end{singlespace}
\begin{singlespace}
\bibitem[vdB]{vdB}
M. van den Bergh, {\sl Linearisations of Binary and Ternary Forms}, J. Algebra \textbf{109} (1987), p. 172-183
\end{singlespace}
\begin{singlespace}
\bibitem[Voi1]{Voi1}
C. Voisin, {\sl Green's generic syzygy Conjecture for curves of even genus lying on a K3
surface}, J. European Math. Soc. \textbf{4} (2002), p. 363�404.
\end{singlespace}
\begin{singlespace}
\bibitem[Voi2]{Voi2}
C. Voisin, {\sl Green's canonical syzygy Conjecture for generic curves of odd genus}, Compositio
Math. \textbf{141} (2005), p. 1163�1190
\end{singlespace}
\end{thebibliography}
\end{document}